\documentclass[11pt]{amsart}
\usepackage[margin=3.14cm]{geometry}
\usepackage{mathrsfs,amsmath,amsfonts,amssymb,amsthm}
\usepackage{enumitem}
\usepackage{graphicx}
\usepackage[dvipsnames,svgnames,table]{xcolor}
\usepackage[linktocpage=true,colorlinks=true,linkcolor=RubineRed!85!black,citecolor=ForestGreen,urlcolor=green,pdfborder={0 0 0}]{hyperref}

\theoremstyle{plain}
\newtheorem{theorem}{Theorem}[section]
\newtheorem{corollary}[theorem]{Corollary}
\newtheorem{proposition}[theorem]{Proposition}
\newtheorem{lemma}[theorem]{Lemma}
\theoremstyle{definition}
\newtheorem{definition}[theorem]{Definition}
\newtheorem{remark}[theorem]{Remark}
\numberwithin{equation}{section}

\def \zero {{\boldsymbol 0}}
\def \ee {{\boldsymbol e}}
\def \Ck {C_\ell}
\def \R {\mathbb{R}}
\def \LL {\mathscr{L}}
\def \OO {\mathcal{O}}
\def \RR {\mathcal{R}}
\def \QQ {\mathcal{Q}}
\def \PP {\mathcal{P}}
\def \aa {\mathbf{a}}
\def \bb {\mathbf{b}}
\def \cc {\mathbf{c}}
\def \hh {\mathbf{h}}
\def \rr {{\tilde{r}}}
\def \zz {{\tilde{z}}}
\def \tt {{\tilde{t}}}
\def \xx {\tilde{x}}
\def \vv {\tilde{v}}

\begin{document}
\title[Sharp boundary regularity for hypoelliptic kinetic equations]
{Sharp boundary regularity properties for hypoelliptic kinetic equations}	
\author{Yuzhe Zhu} 
\date{September 1, 2025}
\thanks{The author thanks Cl\'ement Mouhot and Luis Silvestre for helpful discussions.}

\begin{abstract}
We establish sharp boundary regularity results for solutions to kinetic Fokker-Planck equations under prescribed inflow boundary conditions, providing precise quantification of the boundary hypoelliptic regularization effect. For equations with rough coefficients, we characterize the behaviours for solutions on grazing and incoming boundaries. In particular, in the absence of influxes and sources, an explicit exponential infinite-order vanishing estimate is derived near incoming boundaries. When the coefficients are regular, we obtained the optimal H\"older regularity on grazing boundaries and general Schauder-type estimates away from them. 
\end{abstract}
\maketitle

\hypersetup{bookmarksdepth=2}
\setcounter{tocdepth}{2}
\tableofcontents

\section{Introduction}
We study the kinetic Fokker-Planck equation in nondivergence form, given by
\begin{align}\label{KFP}
\begin{aligned}
\partial_tf+v\cdot\nabla_xf=A:D^2_v f+B\cdot\nabla_vf+S. 
\end{aligned}
\end{align}
The unknown $f=f(z)$ is defined for $z:=(t,x,v)\in(T_0,T]\times\Omega\times\R^d$ with $T_0<T$ and a bounded domain $\Omega\subset\R^d$. The equation is supplemented with the prescribed inflow boundary condition
\begin{align*}
f|_{\Sigma_-}=f_b, 
\end{align*}
where the incoming boundary set $\Sigma_-$ is defined below and $f_b=f_b(z)$ is a given function. The coefficients of the equation consist of a $d\times d$ real symmetric matrix $A=A(z)$ and a $d$-dimensional vector $B=B(z)$, and the source term $S=S(z)$ is a given scalar function. We assume the existence of constants $\Lambda>\lambda>0$ such that for any $z\in(T_0,T]\times\Omega\times\R^d$, 
\begin{align}\label{Elliptic}
\left\{\begin{aligned}
\ &\lambda I_d\le A(z)\le\Lambda I_d, \\
\ &\!|B(z)|\le\Lambda\left(1+|v|^2\right), \\
\ &\;\! S(z) {\ \, \rm essentially\ bounded}. 
\end{aligned}\right. 
\end{align}	
Let $n_x\in\R^d$ denote the unit outward normal vector at $x \in \partial\Omega$. We decompose the phase boundary 
\begin{align*}
\Sigma:=(T_0,T]\times\partial\Omega\times\R^d
\end{align*}
of $(T_0,T]\times\Omega\times\R^d$ into three subsets, namely, 
\begin{align*}
\begin{alignedat}{2}
{\rm Outgoing\ boundary}&\quad \Sigma_+&\,:=\,&\left\{(t,x,v)\in\Sigma :\, n_x\cdot v>0\right\},\\
{\rm Incoming\ boundary}&\quad \Sigma_-&\,:=\,&\left\{(t,x,v)\in\Sigma :\, n_x\cdot v<0\right\},\\
{\rm Grazing\ boundary}&\quad \Sigma_0&\,:=\,&\left\{(t,x,v)\in\Sigma :\, n_x\cdot v=0\right\}.
\end{alignedat}
\end{align*}

For convenience in stating the main results, we define the local domain $G_r(z_0)$ by restricting $Q_r(z_0)$ to the spatial domain $\Omega$,
\begin{align*}
G_r(z_0):=\{(t,x,v)\in Q_r(z_0):\,x\in\overline{\Omega}\}, 
\end{align*}
where $Q_r(z_0)$ is the kinetic cylinder centered at $z_0=(t_0,x_0,v_0)\in\R^{1+2d}$ with radius $r>0$,
\begin{align*}
Q_r(z_{0}):=\left\{ (t,x,v): \, t_{0} - r^{2} < t \le t_{0}, \, |x - x_{0} - (t - t_{0})v_{0}| < r^{3},\, |v - v_{0}| < r\right\}.
\end{align*}

\subsection{Main results}
Our objective is to analyse the boundary regularity of solutions to the hypoelliptic kinetic equation of the form \eqref{KFP}, aiming for sharp results. This equation features a combination of transport that mixes the variables $x,v$ with diffusion acting only in $v$, formulated in non-divergence form, under the prescribed inflow boundary condition. Given its structure, we work with classical solutions, meaning the solution $f$, for which $(\partial_t+v\cdot\nabla_x)f$ and $D_v^2f$ are continuous in the interior, that satisfies \eqref{KFP} pointwise and matches the prescribed boundary data continuously. 

The regularity of solutions is intimately linked to the regularity of the coefficients and the boundary information. In the context of equations with merely bounded and measurable coefficients, we establish H\"older regularity for solutions on the grazing and incoming boundaries, $\Sigma_0\cup\Sigma_-$. 

\begin{theorem}[H\"older estimate on $\Sigma_0\cup\Sigma_-$]\label{rough-holder}
Let $\partial\Omega\in C^{1,1}$, $z_0\in\Sigma_0\cup\Sigma_-$, and $f_b\in C^{\alpha_b}(\Sigma_-)$ for some $\alpha_b\in(0,1)$. Assume that $f$ is a solution of \eqref{KFP} in $G_2(z_0)$ subject to \eqref{Elliptic} with $f|_{\Sigma_-}=f_b$. Then, there exists $\alpha\in(0,1)$ depending only on $d,\alpha_b,\lambda,\Lambda,\Omega$ such that for any $z\in G_1(z_0)$, 
\begin{align*}
|f(z_0)-f(z)|\le C|z_0-z|^\alpha\big(\|f_b\|_{C^{\alpha_b}(Q_2(z_0)\cap\Sigma_-)} + \|f\|_{L^\infty(G_2(z_0))}+ \|S\|_{L^\infty(G_2(z_0))}\big), 
\end{align*}
where the constant $C>0$ depends only on $d,\alpha_b,\lambda,\Lambda,\Omega$ and $|z_0|$. 
\end{theorem}

In the absence of influxes, we obtain a gradient estimate for solutions to \eqref{KFP} on the incoming boundary $\Sigma_-$. When the source term also vanishes, we further derive an explicit exponential estimate demonstrating infinite-order vanishing behaviour for solutions as they approach $\Sigma_-$ from the inside. A distinguishing feature of these asymptotic estimates is their independence from the lower bound of the ellipticity constant $\lambda$ for the diffusion matrix $A$ in \eqref{Elliptic}. 

\begin{theorem}[Asymptotic estimates on $\Sigma_-$]\label{rough-asmp}
Let $\partial\Omega\in C^{1,1}$ and $z_0=(t_0,x_0,v_0)\in\Sigma_-$. Assume that $f$ is a solution of \eqref{KFP} subject to \eqref{Elliptic} in $G_2(z_0)$ with $f|_{\Sigma_-}=0$.
\begin{enumerate}[label=(\roman*), leftmargin=*]
\item\label{rough-asmp1}(Gradient estimate) For any $z\in G_1(z_0)$, we have 
\begin{align*}
|f(z)|\le \frac{C_\star|z-z_0|}{\,|v_0\cdot n_{x_0}|^3} \left(\|f\|_{L^\infty(G_2(z_0))}+ \|S\|_{L^\infty(G_2(z_0))}\right). 
\end{align*}

\item\label{rough-asmp2}(Infinite-order vanishing)
If additionally $S=0$ in $G_2(z_0)$, then for any $x\in\R^d$ such that $(t_0,x,v_0)\in G_1(z_0)$ and $0<{\rm distance\;\!}(x,\partial\Omega)\le |v_0\cdot n_{x_0}|^3$, we have
\begin{align*}
|f(t_0,x,v_0)|\le\|f\|_{L^\infty(G_2(z_0))}\;\!\exp\bigg(1-\frac{\left|n_{x_0}\cdot v_0\right|^3}{C_\star\,{\rm distance\;\!}(x,\partial\Omega)}\bigg). 
\end{align*}
\end{enumerate}
Here the constant $C_\star>0$ depends only on $d,\Lambda,\Omega$ and $|z_0|$, independent of $\lambda$. 
\end{theorem}

Achieving higher-order regularity for solutions requires a certain amount of regularity in the coefficients and sources, as well as the boundary information. A standard technique for simplifying boundary value problems involves reducing suitably regular prescribed boundary data to zero influx by subtraction and subsequent modification of the source term. We focus on this regime to show the applicability of general Schauder-type estimates to the hypoelliptic kinetic equation \eqref{KFP} away from the grazing set $\Sigma_0$. For clarity, the precise definition of kinetic H\"older spaces $\Ck^{k+\alpha}$ is provided in Subsection~\ref{k-holder} below. 

\begin{theorem}[Higher-order regularity away from $\Sigma_0$]\label{regular-Schauder}
Assume that $z_0\in\R^{1+2d}$ satisfies $G_{2r}(z_0)\cap\Sigma_0=\emptyset$ for some $r\in(0,1]$. Let $f$ be a solution of \eqref{KFP} subject to \eqref{Elliptic} in $G_{2r}(z_0)$ with $f|_{\Sigma_-}=0$. 
\begin{enumerate}[label=(\roman*), leftmargin=*]
\item\label{regular-Schauder-h}(Schauder estimate) If $\partial\Omega\in\Ck^{6+k+\alpha}$ and $A,B,S\in\Ck^{k+\alpha}(G_{2r}(z_0))$ for some $k\in\mathbb{N}$ and $\alpha\in(0,1)$, then there exists some constant $C_k>0$ depending only on $d,\lambda,\Lambda,\Omega,|z_0|$, $k,\alpha$ and $\Ck^{k+\alpha}(G_{2r}(z_0))$-norms of the coefficients $A,B$ such that
\begin{align*}
r^{k+2+\alpha}\,[f]_{\Ck^{k+2+\alpha}(G_r(z_0))}
\le C_k\big(\|f\|_{L^\infty(G_{2r}(z_0))} + \|S\|_{\Ck^{k+\alpha}(G_{2r}(z_0))}\big).
\end{align*}

\item\label{regular-Schauder-CN}(Cordes-Nirenberg estimate) Let $\partial\Omega\in C^{1,1}$ and $\alpha\in(0,1)$. There exist some constants $\varrho,R_0\in(0,1]$ depending only on $d,\lambda,\Lambda,\Omega,\alpha$ such that if $A$ satisfies 
\begin{align*}
\sup\nolimits_{\{z,z':\,z\in G_{R_0}(z')\subset G_{2r}(z_0)\}}|A(z)-A(z')|\le \varrho, 
\end{align*}
then for some $C>0$ depending only on $d,\lambda,\Lambda,\Omega,|z_0|,\alpha$, we have 
\begin{align*}
r^{1+\alpha}\,[f]_{\Ck^{1+\alpha}(G_r(z_0))}
\le C\big( \|f\|_{L^\infty(G_{2r}(z_0))} + \|S\|_{L^\infty(G_{2r}(z_0))}\big).
\end{align*}
\end{enumerate}
\end{theorem}

The boundary regularization effect is fundamentally limited near the grazing set $\Sigma_0$, where the hypoelliptic mechanism, governed by the interplay between transport and diffusion, diminishes. Despite this, we prove the following optimal $\Ck^\frac{1}{2}$ H\"older regularity everywhere. Here the kinetic H\"older regularity, characterized by the exponent $\frac{1}{2}$, is locally equivalent to $C_{t,\;\!x,\;\!v}^{\frac{1}{4},\frac{1}{6},\frac{1}{2}}$ under the standard notion of anisotropic H\"older spaces. 

\begin{theorem}[Optimal H\"older regularity]\label{regular-optimal}
Let $\partial\Omega\in C^{1,1}$, $z_0\in\R^{1+2d}$, and $A\in\Ck^\alpha(G_2(z_0))$ for some $\alpha\in(0,1)$. Assume that $f$ is a solution of \eqref{KFP} subject to \eqref{Elliptic} in $G_2(z_0)$ with $f|_{\Sigma_-}=0$. Then, we have
\begin{align*}
[f]_{\Ck^{1/2}(G_1(z_0))}\le C\big(\|f\|_{L^\infty(G_2(z_0))}+\|S\|_{L^\infty(G_2(z_0))}\big), 
\end{align*}
where the constant $C>0$ depends only on $d,\lambda,\Lambda,\Omega,|z_0|,\alpha$ and $[A]_{\Ck^\alpha(G_2(z_0))}$.    
\end{theorem}

\subsection{Background and comments}
\subsubsection{Diffusive kinetic equations}
The evolution of a particle system at the mesoscopic level can be described by a distribution function $f(t,x,v)$ over time $t$ in the phase space $\Omega\times\R^d$, which encodes the macroscopic variable $x\in\Omega\subset\R^d$ for position and the microscopic variable $v\in\R^d$ for velocity. The governing kinetic equations typically involve two components: particle transport, where free streaming under Hamiltonian dynamics is captured by the operator $\partial_t+v\cdot\nabla_x$, and particle interactions, modelled by a collision operator. In the presence of long-range interactions, collisions are dominated by small-angle deflections, which induce small changes in particle velocity, and cumulatively, result in a diffusion effect in velocity space. This motivates kinetic models of the form \eqref{KFP}, featuring a second-order elliptic operator in the velocity variable on the right-hand side. Physically refined models that account for long-range interactions include the Landau equation, which exhibits similar diffusion in velocity driven by Coulomb interactions, and the non-cutoff Boltzmann equation, whose collision kernel retains angular singularities and induces an integral diffusion in velocity. One may refer to \cite{Villani-2006,MV2015,Mouhot-2018} for further discussions. 

\subsubsection{Interior hypoelliptic regularity}
An aspect of understanding the physical behaviours of such systems lies in the regularity properties of solutions to the diffusive kinetic equations. Despite its diffusion degeneracy in the space variable, a primary regularization mechanism, known as {\em hypoellipticity}, emerges from the combination of velocity-space diffusion by the collision operator and phase-space mixing by the transport operator. The theory of hypoelliptic second-order operators developed in \cite{Ho} reveals a foundational structure underlying this phenomenon. Specifically for \eqref{KFP}, the commutator identity $[\nabla_v,\partial_t+v\cdot\nabla_x]=\nabla_x$ illustrates how derivatives in the velocity space, where diffusion occurs, can generate derivatives in the space variable through the transport dynamics. This mechanism leads to a smoothing effect in $x$ despite the lack of direct diffusion in that direction. It turns out that if the coefficients and the source of \eqref{KFP} lie in $C^\infty$, then the solution is $C^\infty$ on the interior of the domain. 

Recent advances in the study of hypoelliptic regularity for a broad class of kinetic equations, whether linear or nonlinear, have yielded results ranging from H\"older continuity to $C^\infty$ smoothness. Broadly speaking, these developments show that any bounded, rapidly decaying solutions to the nonlinear collisional models, among which are the Landau and non-cutoff Boltzmann equations, are smooth in the absence of boundaries; see \cite{HS,IS}. The regularity theory of \eqref{KFP} plays a crucial role, as its estimates directly inform and carry over to the analysis of the nonlinear equations; see \cite{GIMV,IM}. For extended overviews, one may consult \cite{MV2015,Mouhot-2018,Silvestre-problems}. 

\subsubsection{Boundary effects}
While hypoellipticity ensures interior smoothness of solutions to the kinetic equations, understanding the behaviours near the boundary remains subtle. A fundamental aspect is that the transport nature $v\cdot\nabla_x$ induces velocity directions that either point into the domain (incoming part $\Sigma_-$) or out of the domain (outgoing part $\Sigma_+$) at the space boundary. Well-posedness of kinetic equations requires prescribing suitable boundary conditions for the incoming particle flux, which vary depending on how particles interact with the boundary. Perfectly elastic bounces give specular reflection, thermal re-emission leads to diffuse reflection, and more general mixtures may incorporate absorption, prescribed fluxes or Maxwell-type accommodation. A more detailed account can be found in \cite{Cerc}. 

The key analytical obstacle to boundary regularity arises from the grazing characteristics (grazing part $\Sigma_0$), where particle velocities are tangent to the boundary: $n_x\cdot v=0$. At such points, the transport field $v\cdot\nabla_x$ carries no normal component, meaning particles interact with the boundary in a way that does not immediately push into or out of the domain. This loss of normal transport often impedes the gain of regularity up to the boundary to a level comparable to that typically achieved for interior regularity; see also \cite{HJV,Zhu} for related discussions. 

\subsubsection{H\"older and asymptotic estimates near the boundary}
We investigate in this work the sharp regularity properties for a class of hypoelliptic kinetic equations in nondivergence form, subject to prescribed inflow boundary conditions. We first characterize boundary H\"older regularity and the asymptotic behaviours of solution to \eqref{KFP} with bounded and measurable coefficients in Theorem~\ref{rough-holder} and Theorem~\ref{rough-asmp}, respectively. Their proofs, given in Section~\ref{sec-rough}, are based on the construction of barriers tailored to the kinetic geometry associated with \eqref{KFP}. This approach exploits the transport part of the equation to compensate for diffusion degeneracy and is loosely inspired by hypocoercivity techniques for diffusive kinetic equations (see \cite{Villani-2006,IM}). Moreover, it furnishes a gradient-estimate framework that serves as a foundation for deriving higher-order regularity in Section~\ref{sec-regular}. 

For the corresponding hypoelliptic kinetic equations in divergence form with rough coefficients, the global H\"older regularity has been established in \cite{Sil,Zhu}. Asymptotic behaviour at the incoming boundary similar to that described in part~\ref{rough-asmp2} of Theorem~\ref{rough-asmp}, in the setting without influxes and sources, was studied in \cite{Sil} for divergence form equations and in \cite{HLW} for equations with constant coefficients. Notably, this infinite-order vanishing estimate stands in marked contrast to the classical Hopf lemma for diffusion equations, and is more aptly compared to the asymptotic behaviour near the initial time of solutions to parabolic equations with zero initial data; one may refer to \cite{Aronson} for the exponential infinite-order vanishing phenomenon near $t=0$ in the fundamental solutions. 

We emphasize that Theorem~\ref{rough-holder} does not address the regularity of solutions near the outgoing boundary. This omission is due to the intrinsic difficulty in obtaining interior H\"older estimates in the nondivergence form setting of \eqref{KFP}. In particular, a hypoelliptic analogue of the Krylov-Safonov theory remains largely undeveloped; see \cite{Silvestre-problems} for a comprehensive discussion. 

\subsubsection{Higher-order regularity results}
In Section~\ref{sec-regular}, we consider the case where the coefficients in \eqref{KFP} are regular. We establish higher-order regularity for solutions to \eqref{KFP} away from the grazing set, as presented in Theorem~\ref{regular-Schauder}. This finding is consistent with the general Schauder theory for elliptic equations (see \cite{GT}). A case of particular interest arises when the leading coefficients, namely the entries of the diffusion matrix $A$ in \eqref{KFP}, possess only small oscillations at small scales, which are not necessarily continuous. Under this small-oscillation assumption, we obtain $\Ck^{1+\alpha}$ estimates for solutions to \eqref{KFP} for every $\alpha\in(0,1)$, as stated in part~\ref{regular-Schauder-CN} of Theorem~\ref{regular-Schauder}, in the vein of the classical Cordes-Nirenberg estimates (see for instance \cite{Caf1}). 

The regularity of solutions to \eqref{KFP} up to the incoming boundary $\Sigma_-$ follows from the gradient estimates established in Section~\ref{sec-rough}, upon recognizing that a perturbative argument reduces the study to the constant-coefficients scenario. At the outgoing boundary $\Sigma_+$, the regularity issues turn out to be essentially equivalent to those encountered in the interior of the domain; see for instance \cite{Zhu}. Intuitively, since particles exiting the domain do not require any prescribed conditions, the regularity of solutions near $\Sigma_+$ is governed by the interior mechanisms of the equation. As a result, the higher-order regularity of solutions to \eqref{KFP} with regular coefficients and boundary information is ensured throughout the bounded domain, with the notable exception of the grazing set $\Sigma_0$. 

\subsubsection{Optimal H\"older regularity near the grazing boundary}
Theorem~\ref{regular-optimal} addresses the optimal $\Ck^{1/2}$ kinetic H\"older regularity for solutions to \eqref{KFP} up to the grazing set, under the assumption that only the leading coefficients are H\"older continuous. The sharpness of the $\Ck^{1/2}$ regularity is illustrated by the counterexamples constructed in \cite{GJW}, which demonstrate that even for the one-dimensional stationary model with constant coefficients, solutions may fail to exhibit any higher-order regularity. 

On the constructive side, the explicit example from \cite{GJW}, which we analyse further in Section~\ref{sec-regular}, serves as a key ingredient in constructing a suitable barrier. This kind of barrier function enables us to derive matching regularity bounds for \eqref{KFP} by precisely capturing its boundary behaviours dictated by a simplified model scenario. One may also find a different aspect of related asymptotic features in the study of sharp time decay estimates for the constant-coefficient equation in \cite{HLW}. 

\subsubsection{Further remarks}
While our analysis is carried out at the linear level, it directly informs the nonlinear setting. Based on the global a priori H\"older estimates developed over the course of \cite{GIMV,Sil,Zhu}, Theorem~\ref{regular-optimal} shows that the $\Ck^{1/2}$ kinetic H\"older space characterizes the optimal regularity class for solutions to the nonlinear Landau equation with prescribed inflow boundary conditions. It also follows from the boundary Schauder estimates in Theorem~\ref{regular-Schauder} that $\Ck^{5/2}$ regularity up to the boundary can be attained by such solutions away from the grazing set.

We conclude by noting that the mathematical treatment of the regularity problem for various nonlocal reflection boundary conditions, including general diffuse reflection, has been found, in substance, to be analogous to that for prescribed inflow boundary conditions; see \cite{Zhu}. In addition, the optimal boundary regularity in the case of specular reflection was recently studied in \cite{ROW}, offering another perspective on the boundary hypoelliptic regularization effect. 

\subsection{Organization of the paper}
Section~\ref{sec-pre} contains the preliminaries, including the basic notation, the notions of kinetic H\"older spaces, and the boundary flattening procedure. We establish Theorem~\ref{rough-holder} and Theorem~\ref{rough-asmp} in Section~\ref{sec-rough}, which concern the boundary H\"older regularity and asymptotic estimates, respectively. Higher-order regularity away from grazing boundaries, as stated in Theorem~\ref{regular-Schauder}, and the optimal H\"older regularity near grazing boundaries, given in Theorem~\ref{regular-optimal}, are proved in Section~\ref{sec-regular}. 

\section{Preliminaries}\label{sec-pre}
This section is devoted to introducing the basic notation and kinetic H\"older spaces, as well as outlining the boundary flattening argument.

\subsection{Notation}
For clarity and reference, the main notation is enumerated below. Further notions concerning kinetic H\"older spaces are introduced in the next subsection. 
\begin{itemize}[leftmargin=*]
\item For $z=(t,x,v)\in\R\times\R^d\times\R^d$, the \emph{kinetic scaling} $S_r$ with $r>0$ is defined by 
\begin{align*}
S_r(z):= (r^2t,\,r^3x,\,rv). 
\end{align*}
\item The \emph{Galilean group operation} of $z=(t,x,v)$ with respect to a base point $z_0=(t_0,x_0,v_0)\in\R\times\R^d\times\R^d$ is given by
\begin{align*}
z_0\circ z:=(t+t_0,\,x+x_0+tv_0,\,v+v_0).
\end{align*}
Under this group operation, the inverse of $z$ takes the form 
\begin{align*}
z^{-1}:=(-t,-x+tv,-v).
\end{align*}
\item The \emph{kinetic cylinder} centered at $z_0=(t_0,x_0,v_0)$ with radius $r>0$ is defined as
\begin{align*}
Q_r(z_{0}):=&\left\{ z_{0} \circ S_r(z): \, z \in (-1,0]\times B_1(0)\times B_1(0)\right\}. 
\end{align*}
\item We adopt the standard multi-index notation. For a multi-index $m=(m_1,\ldots,m_N)\in\mathbb{N}^N$, 
\begin{align*}
|m|:=\sum\nolimits_{i=1}^Nm_i,\quad\  m!:=\prod\nolimits_{i=1}^Nm_i!. 
\end{align*}
For $u=(u_1,\ldots,u_N)\in\R^N$, we write 
\begin{align*}
u^m:=\prod\nolimits_{i=1}^Nu_i^{m_i},\quad\ D_u^m:=\prod\nolimits_{i=1}^N\partial_{u_i}^{m_i}. 
\end{align*}
\item We consider the spatially restricted cylinder over the domain $\Omega$, defined by
\begin{align*}
G_r(z_0):=\{(t,x,v)\in Q_r(z_0):\,x\in\overline{\Omega}\}. 
\end{align*}
When the cylinder is centered at the origin of $\R^{1+2d}$, we simply write $G_r$. 
\item The $d\times d$ identity matrix is denoted by $I_d$. 
\item For $i\in\{1,\ldots,d\}$, let $\ee_i$ denote the standard orthonormal basis of $\R^d$. 
\item We write $\R_+=[0,\infty)$, $\R_-=(-\infty,0]$, and define the half-space $\mathbb{H}_-^d=\R^{d-1}\times\R_-$. 
\item The zero vectors in $\R^d$ and $\R^{d-1}$ are denoted by $\zero$ and $\zero'$, respectively. 
\item Given a vector $u=(u_1,\ldots,u_d)\in\R^d$, we write $u'=(u_1,\ldots,u_{d-1})$ for its first $d-1$ components. 
\item In the half-space setting, we define the phase domain and its boundary by
\begin{align*}
\OO:=\mathbb{H}_-^d\times\R^d,\quad \Gamma:=\partial\OO=\R^{d-1}\times\{0\}\times\R^d.
\end{align*} 
The associated time-dependent sets are given by  
\begin{align*}
\OO_T:=(-\infty,0]\times\OO,\quad \Sigma:=(-\infty,0]\times\Gamma. 
\end{align*} 
We further specify the outgoing, incoming, and grazing subsets of the boundary, along with their time-dependent counterparts, as 
\begin{align*}
\begin{alignedat}{2}
\Gamma_\pm&=\left\{(x,v)\in\Gamma:\,x_d=0,\ \,\pm v_d>0\right\},\quad &\Sigma_\pm&=(-\infty,0]\times\Gamma_\pm,\\
\Gamma_0&=\left\{(x,v)\in\Gamma:\,x_d=0,\ \, v_d=0\right\},\quad &\Sigma_0&=(-\infty,0]\times\Gamma_0.
\end{alignedat}
\end{align*}
\item Given coefficients satisfying \eqref{Elliptic}, we abbreviate the hypoelliptic operator 
\begin{align*}
\LL:=\partial_t+v\cdot\nabla_x- A:D^2_v-B\cdot\nabla_v.
\end{align*}
For the purpose of illustrating the constant-coefficient setting, we consider
\begin{align*}
\LL_0:=\partial_t+v\cdot\nabla_x- \Delta_v.
\end{align*}
\item We adopt the bracket convention $\langle\cdot\rangle:=(1+|\cdot|^2)^{1/2}$. 
\item To account for the possible velocity-dependent growth of the bound on the coefficient $B$ (see \eqref{Elliptic} and \eqref{Elliptic-v}), we introduce the shorthand $C_B:=C\langle v_0\rangle^2$, where $C\ge 1$ is a constant selected so that $C_B\ge\|B\|_{L^\infty(G_1(z_0))}$ for a prescribed point $z_0=(t_0,x_0,v_0)$ which serves as the center of a cylinder. 
\item A quasi-distance function $\rho:\PP\to\R_+$, defined for $(x,v)\in\PP\subset\R^d\times\R^d$, is specified in Lemma~\ref{hypodist1}. The associated time-dependent function $\rho_t(x,v)=\rho(x^t,v)$, for $(t,x,v)\in\PP_T$, is introduced in \S~\ref{evolu} as the composition of $\rho$ with the spatial translation $x\mapsto x^t$ in its argument. 
\item Let $\Omega_{t,x}\subset\R\times\R^d$ and $\Omega_v\subset\R^d$ be domains with Lipschitz boundaries, and consider $G:=\Omega_{t,x}\times\Omega_v$. The \emph{kinetic boundary} of $G$ is defined as
\begin{align*}
\partial_{\rm kin} G:=\{(t,x,v)\in\partial\Omega_{t,x}\times\Omega_v:(1,v)\cdot n_{t,x}<0\}\cup(\Omega_{t,x}\times\partial\Omega_v), 
\end{align*}
where $n_{t,x}\in\R^{1+d}$ denotes the unit outward normal vector at $(t,x)\in\partial\Omega_{t,x}$. 
\item We write $Y\lesssim Z$ if $Y\le CZ$ for some constant $C>0$ independent of the parameters of interest in a given context. The implicit constant $C$ is understood to depend only on the parameters specified in the corresponding statement, such as $d,\lambda,\Lambda,\Omega$. The notation $Y\approx Z$ indicates that both $Y\lesssim Z$ and $Z\lesssim Y$ hold. 
\end{itemize}

\subsection{Kinetic H\"older spaces}\label{k-holder}
We present the framework of kinetic H\"older spaces and the associated kinetic geometry, tailored to the kinetic scaling and the Galilean group operation. This provides the natural language for formulating regularity estimates for solutions to kinetic hypoelliptic equations of type~\eqref{KFP}. Comprehensive expositions, as well as rigorous justifications, can be found in \cite{IM,IS-Schauder}. 

\subsubsection{Kinetic H\"older continuity}
For a monomial $m(z)=z^l$ with $z=(t,x,v)\in\R^{1+2d}$, written in multi-index notation as
\begin{align*}
m(t,x,v)=t^{l_t}x^{l_x}v^{l_v},\quad l=(l_t,l_x,l_v)\in\mathbb{N}\times\mathbb{N}^d\times\mathbb{N}^d,
\end{align*}
its \emph{kinetic degree} is defined by
\begin{align*}
\deg_{\rm kin}(m):=|l|_{\rm kin}:=2l_t +3|l_x|+|l_v|.
\end{align*}
Every polynomial $p\in\R[t,x,v]$ can be uniquely expressed as a linear combination of such monomials, and its kinetic degree $\deg_{\rm kin}(p)$ is defined as the maximum among their degrees. 

\begin{definition}
Let $\beta>0$ and let $G\subset\R\times\R^d\times\R^d$ be an open set. A function $f:G\to\R$ is said to be \emph{$\Ck^\beta$-continuous} at a point $z_0\in G$, if there exists a polynomial $p_0\in\R[t,x,v]$ with $\deg_{\rm kin}(p_0)<\beta$ and a constant $C>0$ such that for any $r>0$,  
\begin{align}\label{holdercondition}
\|f-p_0\|_{L^\infty(Q_r(z_0)\cap G)}\le C\;\!r^\beta. 
\end{align}
We say $f\in \Ck^\beta(G)$ if $f$ is $\Ck^\beta$ at every $z_0\in G$. The smallest admissible constant $C$ in \eqref{holdercondition} is denoted by $[f]_{\Ck^\beta(G)}$, called the $\Ck^\beta$ semi-norm. The corresponding norm is
\begin{align*}
\|f\|_{\Ck^\beta(G)}:=\|f\|_{L^\infty(G)}+[f]_{\Ck^\alpha(G)}. 
\end{align*}
\end{definition}

\subsubsection{Kinetic gauge}
Define the \emph{kinetic gauge} $\|\cdot\|$ for $z=(t,x,v)\in\R^{1+2d}$ by  
\begin{align*}
\|z\|:=\max\big\{|t|^{1/2},|x|^{1/3},|v|\big\}. 
\end{align*}   
With this notation, the H\"older condition \eqref{holdercondition} is equivalent to 
\begin{align*}
|f(z)-p_0(z)|\le C\;\!\|z_0^{-1}\circ z\|^\beta {\quad\rm for\ all\ }z\in G. 
\end{align*}
It should be observed that the kinetic gauge is homogeneous under the kinetic scaling, satisfies the triangle inequality with respect to the Galilean group operation, and is stable under the group inversion, namely, for any $z,z_0\in\R^{1+2d}$ and $r>0$, 
\begin{align*}
&\|S_r(z)\|=r\;\!\|z\|,\\
&\|z_0\circ z\|\le \|z_0\|+\|z\|,\\
&\,2^{-1/3}\;\!\|z\|\le\|z^{-1}\|\le 2^{1/3}\;\!\|z\|.
\end{align*} 

\subsubsection{Differentiation and polynomial expansion}\label{d-poly}
Similarly to monomials, we assign a kinetic order to differential operators. Consider the left-invariant homogeneous differential operator
\begin{align*}
D^l:=(\partial_t+v\cdot\nabla_x)^{l_t}D_x^{l_x}D_v^{l_v},\quad l=(l_t,l_x,l_v)\in\mathbb{N}\times\mathbb{N}^d\times\mathbb{N}^d. 
\end{align*}
Its \emph{kinetic order} is also regarded as $|l|_{\rm kin}=2l_t +3|l_x|+|l_v|$, and it satisfies the left-invariant property that for any smooth function $f:\R^{1+2d}\to\R$ and any fixed $z_0\in\R^{1+2d}$, 
\begin{align*}
D^l[f(z_0\circ z)]=(D^l f)(z_0\circ z). 
\end{align*}
In particular, this shows that the structure of the equation~\eqref{KFP} is left-invariant under the Galilean group operation. 

The regularity of a function can be characterized either directly in the kinetic H\"older spaces introduced earlier, or equivalently through the regularity of its derivatives. This consistency means that $f\in \Ck^\beta$ if and only if $D^lf\in\Ck^{\beta-|l|_{\rm kin}}$ for every $l$ with $|l|_{\rm kin}<\beta$. 

Note that the polynomial $p_0(z)$ appearing in \eqref{holdercondition} can be obtained from the truncated Taylor expansion of $f$ around $z_0$, which takes the form 
\begin{align*}
T_{z_0,\beta}[f](z):= \sum\nolimits_{|l|_{\rm kin}<\beta} \frac{D^lf(z_0)}{l!}\,\big(z_0^{-1}\circ z\big)^l. 
\end{align*}
Applying the left-invariant operator $D^l$ to the polynomial expansion of $f$ commutes with the expansion up to a degree shift, that is, 
\begin{align*}
D^lT_{z_0,\beta}[f]=T_{z_0,\beta-|l|_{\rm kin}}[D^lf]. 
\end{align*}

\subsection{Boundary flattening procedure}\label{BFP}
All of the analysis to be presented in the subsequent sections is carried out essentially in the context of the half-space problem. The reduction to the half-space setting is achieved through a boundary flattening procedure described in this subsection, which localizes the analysis near the boundary and transforms the original domain into a simpler geometry where the boundary becomes flat and coincides with a coordinate hyperplane. In these new coordinates, \eqref{KFP} is reformulated with modified coefficients, while its essential structural features, namely the transport and velocity diffusion, are preserved. 

Let $z_0=(t_0,x_0,v_0)\in\Sigma$. Suppose that $\partial\Omega$ is of class $C^{k,\alpha}$ for some $k\ge1$ and $\alpha\in(0,1]$. Then, near $x_0$, the boundary $\partial\Omega$ can be flattened by a $C^{k,\alpha}$-diffeomorphism constructed from a $C^{k,\alpha}$-function $\mathscr{P}:\R^{d-1}\to\R$ and a constant $L\in(0,1]$, given by the map 
\begin{align*}
P:B_L(x_0)\cap\overline{\Omega}\ & \to\ \mathbb{H}_-^d=\R^{d-1}\times\R_-, \\
x=(x',x_d)\ & \mapsto\ y=(y',y_d)=(x',x_d-\mathscr{P}(x')),
\end{align*}
so that $P|_{\partial\Omega}\cdot\ee_d=0$, and $|P'|+|P'^{-1}|$ is bounded in $B_L(x_0)$. Here $P'$ denotes the Jacobian matrix of $P$, which can be explicitly expressed as $P'=\left(\begin{smallmatrix}I_{d-1} & 0\\-D\mathscr{P} & 1\end{smallmatrix}\right)$. Consider the phase-domain transformation $\mathcal{E}$ defined on a neighborhood $\mathcal{U}$ of $z_0$,
\begin{align*}
\mathcal{E}:(t,x,v)\ \mapsto\ (t,y,w):=(t,P(x),P'(x)v). 
\end{align*}
Given that $f:\mathcal{U}\to\R$, we define the function $\hat{f}:\mathcal{E}(\mathcal{U})\to\R$ by
\begin{align*}
\hat{f}(t,y,w):=f\circ\mathcal{E}^{-1}(t,y,w)=f(t,x,v) . 
\end{align*}
We observe that the transport term transforms as 
\begin{align*}
\partial_tf+v\cdot\nabla_xf = \partial_t\hat{f}+P'v\cdot\nabla_y\hat{f}+\sum\nolimits_{i,j=1}^d v_i\;\!v_j\;\!\partial_{x_i}\partial_{x_j}P\cdot\nabla_v\hat{f}. 
\end{align*}
Since $P$ is independent of $v$, the diffusion and drift terms transform as  
\begin{align*}
A:D_v^2f+B\cdot\nabla_vf = (P'AP'^T):D_w^2\hat{f}+P'B\cdot\nabla_w\hat{f}. 
\end{align*}
In the new coordinate system $(t,y,w)$, the incoming, outgoing, and grazing parts of the phase boundary are mapped to 
\begin{align*}
\mathcal{E}(\Sigma_\pm\cap\mathcal{U})&=\{y_d=0,\,\pm w_d>0\}\cap\mathcal{E}(\mathcal{U}),\\
\mathcal{E}(\Sigma_0\cap\mathcal{U})&=\{y_d=0,\,w_d=0\}\cap\mathcal{E}(\mathcal{U}). 
\end{align*}
In other words, the boundary flattening procedure reconfigures the problem into a half-space geometry, providing a convenient framework for the subsequent analysis.  

\begin{lemma}\label{flatten}
Let $\partial\Omega\in C^{1,1}$, $z_0=(t_0,x_0,v_0)\in\Sigma$ and $f_b:\Sigma_-\to\R$. The phase-domain transformation $\mathcal{E}$, defined on some neighbourhood $\mathcal{U}$ of $z_0$, ensures that if $f$ satisfies \eqref{KFP} subject to \eqref{Elliptic} in $\mathcal{U}$ and $f=f_b$ on $\Sigma_-\cap\mathcal{U}$, then $\hat{f}=f\circ\mathcal{E}^{-1}$ satisfies 
\begin{align*}
\left\{\begin{aligned}
\ &\partial_t\hat{f}+w\cdot\nabla_y\hat{f}
=\hat{A}:D_w^2\hat{f} +\hat{B}\cdot\nabla_w\hat{f}+\hat{S} {\quad\rm in\ }\mathcal{E}(\mathcal{U}), \\
\ &\,\hat{f}= f_b\circ\mathcal{E}^{-1}{\quad\rm on\ } \mathcal{E}(\Sigma_-\cap\mathcal{U}). \\
\end{aligned}\right. 
\end{align*}	
Here the new coefficients and the new source term defined by
\begin{align*}
\hat{A}\circ\mathcal{E}=P'AP'^T,\quad 
\hat{B}\circ\mathcal{E}=P'B-v^{\otimes2}\!:\!D^2P,\quad 
\hat{S}\circ\mathcal{E}=S. 
\end{align*}
also satisfies \eqref{Elliptic}, in the sense that there is some constant $C_\Omega\ge1$ depending only on $d$ and $\Omega$ such that over the region $\mathcal{E}(\mathcal{U})$, 
\begin{align}\label{Elliptic-v}
\left\{\begin{aligned}
\ &C_\Omega^{-1}\lambda I_d\le \hat{A}\le C_\Omega\,\Lambda I_d, \\
\ &|\hat{B}|\le C_\Omega\,\Lambda\langle v_0\rangle^2.
\end{aligned}\right. 
\end{align}
In addition, if $A\in\Ck^\alpha$ for some $\alpha\in(0,1)$, then $\hat{A}\in\Ck^\alpha$ as well. If $\partial\Omega\in\Ck^{6+\beta}$ and $A,B,S\in\Ck^\beta$ for some $\beta>0$, then we have $\hat{A},\hat{B},\hat{S}\in\Ck^\beta$. 
\end{lemma}

\begin{remark}\label{remark-reduction}
Without loss of generality, we may further assume that $\mathcal{E}(z_0)=(t_0,y_0,w_0)$ satisfies $t_0=0$, $y_0=\zero$ and $w_0=(\zero',w_0\cdot\ee_d)$. Indeed, for general $\mathcal{E}(z_0)=(t_0,y_0,w_0)$ with $w_0=(w_0',w_0\cdot\ee_d)\in\R^{d-1}\times\R$, we consider the base point $\tilde{z}_0:=(t_0,y_0,w_0',0)$. By means of the Galilean transformation, the function $\hat{f}(\tilde{z}_0\circ\hat{z})$, with respect to the variable $\hat{z}=\tilde{z}_0^{-1}\circ\mathcal{E}(z)\in\tilde{z}_0^{-1}\circ\mathcal{E}(\mathcal{U})$, satisfies the same type of equation as $\hat{f}$. Moreover, we have 
\begin{align*}
\tilde{z}_0^{-1}\circ\mathcal{E}(z_0)=(0,\zero,\zero',w_0\cdot\ee_d), 
\end{align*}
which also ensures that the boundary condition is preserved. 

This technique, simplifying the analysis, is standard in the study of interior regularity; see for instance \cite{GIMV}. In contrast, one cannot impose $w_0\cdot\ee_d=0$ unless $w_0$ lies on the grazing set, as applying the transformation to eliminate the $d$-th component of $w_0$ would distort the boundary condition. 
\end{remark}

\section{Rough coefficients case}\label{sec-rough}
In this section, we prove Theorems \ref{rough-holder} and \ref{rough-asmp} by constructing tailored barriers that capture the boundary behaviour of solutions to \eqref{KFP} with rough coefficients near the boundary. In the light of Lemma~\ref{flatten}, to analyse local properties of solutions to \eqref{KFP}, it suffices to consider the half-space case where, locally, $\Omega=\mathbb{H}_-^d=\{x\in\R^d:x_d\le0\}$. Accordingly, throughout this section we work in the phase domain $\OO=\mathbb{H}_-^d\times\R^d$ with boundary $\Gamma=\R^{d-1}\times\{0\}\times\R^d$. In view of Remark~\ref{remark-reduction}, we henceforth restrict our attention to estimates in a neighbourhood of the boundary point $\zz=(\tt,\xx,\vv)=(0,\zero,\zero',\vv_d)\in\R^{1+2d}$ for some $\vv_d\in\R$. 

\subsection{Setup for the quasi-distance function}
The following lemma characterizes a quasi-distance function, whose structure is reminiscent of ideas arising in the development of hypocoercivity theory for diffusive kinetic equations (see \cite{Villani-2006,IM}).

\begin{lemma}\label{hypodist1}
Let $(\xx,\vv)=(\zero,\zero',\vv_d)\in\Gamma$, and let the constants $\rr,\kappa,\aa,\bb,\cc>0$ such that $\sqrt{\aa\cc}\ge4\bb$ and $\aa\ge4\cc$. The following assertions are valid.
\begin{enumerate}[label=(\roman*), leftmargin=*]
\item\label{hdist1} 
A point $(\xi,\eta)\not\in\OO$ can be defined by setting $\xi=(\zero',\xi_d)$ and $\eta=(\zero',\eta_d)$ such that 
\begin{align}\label{xieta}
\left\{\begin{aligned}
\, &\bb\;\!\xi_d=\cc\;\!(\eta_d-\vv_d)\\
\, &\,\xi_d=\frac{\sqrt{\aa\cc}\;\!\rr}{\sqrt{\aa\cc-\bb^2}} \in \left[\;\!\rr,\frac{4}{3}\rr\;\!\right]. 
\end{aligned}\right. 
\end{align}	
\item\label{hdist2} The function $\rho:\R^d\times\R^d\to\R_+$, defined by
\begin{align*}
\rho(x,v):=\sqrt{\aa|\kappa x'|^2+\cc|v'|^2+\aa X_d^2-2\bb X_dV_d+\cc V_d^2}, 
\end{align*}
satisfies 
\begin{align*}
\left\{\begin{aligned}
\, &\rho(\xx,\vv)=\sqrt{\aa}\;\!\rr,\\
\, &\!\big\{\rho(x,v)<\sqrt{\aa}\;\!\rr\big\}\cap\OO=\emptyset, 
\end{aligned}\right. 
\end{align*}	
where $x'$ and $v'$ denote the first $d-1$ components of $x$ and $v$, respectively, so that $x=(x',x_d)$ and $v=(v',v_d)$, and 
\begin{align*}
(X_d,V_d):=(x_d-\xi_d,v_d-\eta_d). 
\end{align*}
\item\label{hdist3} Over the region 
\begin{align*}
\PP:=\big\{\sqrt{\aa}\;\!\rr\le \rho(x,v)\le3\sqrt{\aa}\;\!\rr\big\}\cap\OO, 
\end{align*}
the ranges of $|x'|,|v'|,|X_d|,|V_d|$ satisfy 
\begin{align}\label{range}
\left\{\begin{aligned}
\, &|X_d|\ge\rr,\\
\, &\max\{\kappa\;\!|x'|,|X_d|\}\le 4\;\!\rr,\\
\, &\max\{|v'|,|V_d|\}\le \sqrt{\frac{12\aa}{\cc}}\;\!\rr. 
\end{aligned}\right. 
\end{align}	
Here the construction of  the region $\PP$ is illustrated in Figure~\ref{img-in}. 
\end{enumerate}
\end{lemma}

\begin{figure}
\includegraphics[width=9.35cm]{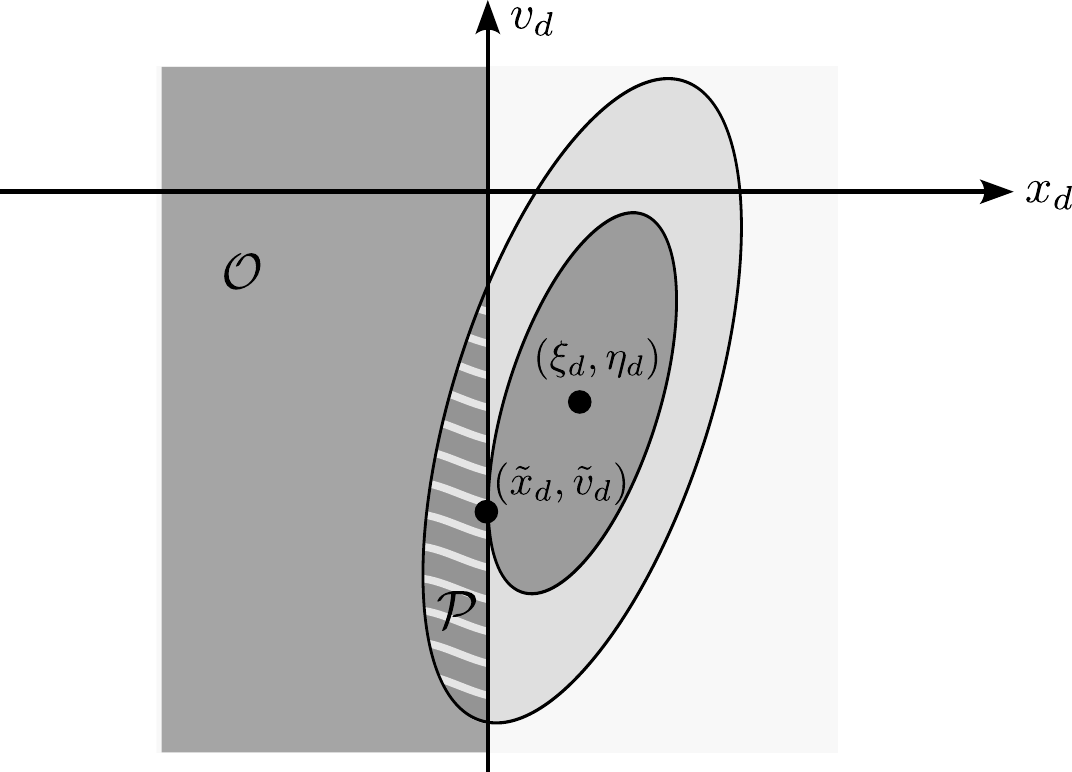}\\\vspace{-0.15cm}
\caption{A section of the constructed region $\PP$.}
\label{img-in}
\end{figure}

\begin{proof}
Parts \ref{hdist1} and \ref{hdist2} can be analysed together. Given $(\xx,\vv)\in\Gamma$ and $\rr,\kappa,\aa,\bb,\cc>0$, over the region $\left\{\rho(x,v)\le\sqrt{\aa}\;\!\rr\right\}$, the quantity $|X_d|$ attains its maximum $\rr_0$ when $\rho(x,v)=\sqrt{\aa}\;\!\rr$ and $\bb X_d=\cc V_d$, that is, $\bb\;\!\xi_d=\cc(\eta_d-\vv_d)$. In this case, we have 
\begin{align*}
\rho(\xx,\vv)=\sqrt{\aa}\;\!\rr,\quad\ 
\xi_d=\rr_0,\quad\ \eta_d-\vv_d=\frac{\bb}{\cc}\;\!\rr_0. 
\end{align*}
These relations allow us to determine $\xi_d$. Furthermore, provided that $\sqrt{\aa\cc}\ge4\bb$, we get  
\begin{align*}
\rr^2\le \rr_0^2 =\xi_d^2
= \frac{\aa\cc\;\!\rr^2}{\aa\cc-\bb^2} \le \frac{16\rr^2}{15}
<\frac{16\rr^2}{9}.
\end{align*}
Taking square roots on both sides implies range of $\xi_d$ in \eqref{xieta}, as well as the lower bound of $|X_d|$ in \eqref{range}. 

For part~\ref{hdist3}, it now suffices to establish the upper bounds in \eqref{range}. We see from the same argument as above that the quantity $|X_d|$ attains its maximum $3\;\!\rr_0$ over $\left\{\rho(x,v)\le3\sqrt{\aa}\;\!\rr\right\}$. Hence, 
\begin{align*}
|X_d|\in\big[\;\!\rr_0,3\;\!\rr_0\big]\subset\big[\;\!\rr,4\;\!\rr\;\!\big] {\quad\rm in\ }\PP. 
\end{align*} 
As for the range of $|V_d|$ over $\big\{\rho(x,v)\le3\sqrt{\aa}\;\!\rr\big\}$, we notice that the smaller eigenvalue $\lambda_0$ of the matrix $\left(\begin{smallmatrix}\aa &-\bb\\-\bb &\cc\end{smallmatrix}\right)$ satisfies 
\begin{align*} 
\lambda_0=\frac{\aa+\cc-\sqrt{(\aa-\cc)^2+4\bb^2}}{2}
=\frac{2(\aa\cc-\bb^2)}{\aa+\cc+\sqrt{(\aa-\cc)^2+4\bb^2}}
\ge \frac{15\aa\cc}{16(\aa+\cc)} 
\ge \frac{3\cc}{4}, 
\end{align*} 
where we used the assumptions $\sqrt{\aa\cc}\ge4\bb$ and $\aa\ge4\cc$ in deriving the last two inequalities above, respectively. It follows that  
\begin{align*}
X_d^2+V_d^2 \le \frac{\rho^2(x,v)}{\lambda_0} \le \frac{9\aa\;\!\rr^2}{\lambda_0}
< \frac{12\aa\;\!\rr^2}{\cc} {\quad\rm in\ }\PP. 
\end{align*}
Finally, since 
\begin{align*}
\aa X_d^2-2\bb X_dV_d+\cc V_d^2\ge0, 
\end{align*}
we derive the ranges of $|x'|$ and $|v'|$ in $\PP$ as claimed.  
\end{proof}

Near the incoming boundary $\Gamma_-=\{(x,v)\in\Gamma:x_d=0,\,v_d<0\}$, the transport term of \eqref{KFP} yields a form of coercivity, as will be demonstrated in subsequent subsections. The following lemma illustrates this effect by providing geometric estimates within $\PP$ relative to its distance from the grazing set $\Gamma_0=\{(x,v)\in\Gamma:x_d=0,\,v_d=0\}$. 

\begin{lemma}\label{hypodist2}
Let $(\xx,\vv)=(\zero,\zero',\vv_d)\in\Gamma_-$, and let the constants $\rr,\kappa,\aa,\bb,\cc>0$ such that $\sqrt{\aa\cc}\ge8\bb$ and $\aa\ge4\cc$. Define the points $\xi,\eta,X,V\in\R^d$, the function $\rho(x, v)$, and the region $\PP$ as in Lemma~\ref{hypodist1}. 

\begin{enumerate}[label=(\roman*), leftmargin=*]
\item\label{vr1} If there holds 
\begin{align}\label{vr}
|\vv_d|\ge \frac{2\bb\;\!\rr}{\cc}, 
\end{align}
then we have $|\eta_d|\ge\frac{9}{20}|\vv_d|$, and
\begin{align*}
\eta_d\left(\aa X_d-\bb V_d\right)\ge\frac{\aa\;\!\rr|\vv_d|}{4} {\quad\rm in\ } \PP. 
\end{align*}

\item\label{vr2} If the following stronger condition holds, namely, 
\begin{align}\label{vrs}
|\vv_d|\ge 8\sqrt{\frac{\aa}{\cc}}\;\!\rr, 
\end{align}
then we have $|v_d|\ge\frac{1}{2}|\vv_d|$, and 
\begin{align*}
v_d\left(\aa X_d-\bb V_d\right)
\ge\frac{\aa\;\!\rr|\vv_d|}{4} {\quad\rm in\ } \PP. 
\end{align*} 
\end{enumerate}
\end{lemma}

\begin{proof}
For part~\ref{vr1}, we use the assumptions~\eqref{vr} and $\sqrt{\aa\cc} \ge 8\bb$, along with the second equality in \eqref{xieta} from Lemma~\ref{hypodist1}, to see that
\begin{align*}
|\vv_d| \ge \frac{2\bb\;\!\rr}{\cc}
\ge\frac{\sqrt{63\aa}\bb\;\!\rr}{4\sqrt{\cc\left(\aa\cc-\bb^2\right)}} 
=\frac{\sqrt{63}\bb\;\!\xi_d}{4\cc}
>\frac{20\bb\;\!\xi_d}{11\cc}. 
\end{align*}
Considering the first equality in \eqref{xieta} and the property of $\Gamma_-$, where $|\vv_d| = -\vv_d > 0$, we deduce that 
\begin{align*}
-\eta_d=|\vv_d|-\frac{\bb\;\!\xi_d}{\cc} \ge \frac{9|\vv_d|}{20}. 
\end{align*} 
See also Figure~\ref{img-in} for a visualization. Noticing $X_d=x_d-\xi_d<0$, we have $\eta_dX_d=|\eta_d||X_d|$. Consequently, 
\begin{align*}
\eta_d\left(\aa X_d-\bb V_d\right)
\ge \frac{9|\vv_d|}{20}(\aa|X_d|-\bb|V_d|). 
\end{align*} 
By \eqref{range} of Lemma~\ref{hypodist1} and the assumption $\sqrt{\aa\cc}\ge8\bb$, we find that 
\begin{align}\label{vrs-ab}
\aa|X_d|-\bb|V_d|\ge \aa\;\!\rr -\frac{\sqrt{12\aa}\bb\;\!\rr}{\sqrt{\cc}}
\ge \bigg(1-\frac{\sqrt{3}}{4}\bigg)\aa\;\!\rr
>\frac{5\aa\;\!\rr}{9} {\quad\rm in\ } \PP. 
\end{align} 
Combining the two estimates above yields the lower bound in part~\ref{vr1}.

For part~\ref{vr2}, we see from the definition of $V_d$ and the triangle inequality that 
\begin{align*}
|v_d|=|V_d+\eta_d| \ge|\vv_d|-|\vv_d-\eta_d|-|V_d|. 
\end{align*}
By the fact $|\vv_d-\eta_d|=\bb\;\!\xi_d/\cc$ from \eqref{xieta} and the range of $|V_d|$ from \eqref{range}, we obtain
\begin{align*}
|v_d|\ge|\vv_d|-\frac{\bb}{\cc}\;\!\xi_d-\sqrt{\frac{12\aa}{\cc}}\;\!\rr {\quad\rm in\ } \PP. 
\end{align*} 
Applying $\sqrt{\aa\cc}\ge8\bb$ and the stronger assumption \eqref{vrs}, we derive 
\begin{align*}
|v_d| \ge|\vv_d|-\frac{1}{8}\sqrt{\frac{\aa}{\cc}}\;\!\rr -\sqrt{\frac{12\aa}{\cc}}\;\!\rr
>|\vv_d|-4\sqrt{\frac{\aa}{\cc}}\;\!\rr
\ge \frac{|\vv_d|}{2} {\quad\rm in\ } \PP. 
\end{align*} 
This also means $v_d<0$ and $v_dX_d=|v_d||X_d|$ in $\PP$. It then follows from \eqref{vrs-ab} that 
\begin{align*}
v_d\left(\aa X_d-\bb V_d\right)
\ge \frac{|\vv_d|}{2}(\aa|X_d|-\bb|V_d|) > \frac{\aa\;\!\rr|\vv_d|}{4} {\quad\rm in\ } \PP, 
\end{align*} 
which establishes the desired estimate in part~\ref{vr2}. 
\end{proof}

\subsection{Gradient estimates near the incoming boundary}
One of the aims of this subsection is to establish part~\ref{rough-asmp1} of Theorem~\ref{rough-asmp}, which addresses the boundary gradient estimates at $\Sigma_-$ for solutions of \eqref{KFP}. This will be obtained directly from Proposition~\ref{phase-prop} below.

To this end, we introduce the operator, with the coefficients satisfying \eqref{Elliptic}, in the abbreviated form 
\begin{align*}
\LL=\partial_t+v\cdot\nabla_x- A:D^2_v-B\cdot\nabla_v. 
\end{align*}
In view of the range of velocity variable in $\PP$ given by Lemma~\ref{hypodist1}, provided that $\sqrt{\frac{\aa}{\cc}}\;\!\rr\le\langle v_0\rangle$, we can assume that for any $z\in(-\infty,0]\times\PP$, 
\begin{align*}
|B(z)|\le C_B, 
\end{align*}
where $0\le C_B\le C\langle v_0\rangle^2$ for some constant $C\ge 1$ depending only on $d$ and $\Lambda$. Here the velocity center $v_0$ in essence arises from the boundary flattening procedure; see \eqref{Elliptic-v}. By explicitly tracking this dependence on $v_0$, we are able to obtain estimates with quantitative involvement on the velocity center. 

\subsubsection{Time-shifted quasi-distance function}\label{evolu}
For subsequent reference, we fix the notation that will be employed consistently henceforth. Let $(\xx,\vv)=(\zero,\zero',\vv_d)\in\Gamma_-$, and let the function $\rho(x,v)$ and the region $\PP$ be those specified in Lemma~\ref{hypodist1}, associated with parameters $\rr,\kappa,\aa,\bb,\cc>0$ subject to the constraints 
\begin{align}\label{abc} 
\sqrt{\aa\cc}\ge8\bb,\quad\ \aa\ge4\cc,\quad\  \sqrt{\frac{\aa}{\cc}}\;\!\rr\le\langle v_0\rangle. 
\end{align}

In order to incorporate the evolutionary equation, we fix a constant $\hh\ge0$, and in terms of the spatial translations
\begin{align*}
&x^t:=(x',x_d-\hh\;\!\vv_dt)\in\PP,\\ &X_d^t:=X_d-\hh\;\!\vv_dt,
\end{align*}
we define the associated time-dependent function $\rho_t=\rho_t(x,v)$, for $(t,x,v)\in\PP_T$, as 
\begin{align}\label{rho-t}
\begin{aligned}
\rho_t(x,v)&:=\rho(x^t,v)=\sqrt{\aa|\kappa x'|^2+\cc|v'|^2+\aa |X_d^t|^2-2\bb X_d^t\;\!V_d+\cc |V_d|^2}, \\
\PP_T&:=\left\{(t,x,v):\rho_0\le\rho_t(x,v)\le3\rho_0\right\}\cap\OO_T {\quad\rm for\ \ }
\rho_0:=\sqrt{\aa}\;\!\rr, 
\end{aligned}
\end{align}
where we recall that $\OO_T=(-\infty,0]\times\mathbb{H}_-^d\times\R^d$. In this setting, whenever $(t,x,v)\in\PP_T$, the shifted variable $X_d^t$ plays the same role as $X_d$ in Lemmas~\ref{hypodist1} and \ref{hypodist2}. 

\subsubsection{Framework of barrier estimates}
The next lemma, in which $\Phi$ serves as a general candidate for barrier functions, highlights how the transport structure of \eqref{KFP} gives rise to a coercivity property that is instrumental in establishing regularity. It is worth pointing out that the estimate from the lemma below is independent of the lower bound $\lambda$ appearing in the condition \eqref{Elliptic}.  

\begin{lemma}\label{barrier}
Let $(\xx,\vv)=(\zero,\zero',\vv_d)\in\Gamma_-$, and let $\rho_t$ and $\PP_T$ be defined in \eqref{rho-t} associated with $\hh\in(0,\frac{1}{36}]$ and $\rr,\kappa,\aa,\bb,\cc>0$ satisfying \eqref{vrs} and \eqref{abc}. There exists some constant $C_\Lambda>0$ depending only on $d$ and $\Lambda$ such that, for any function $\Phi:\R_+\to\R_+$ satisfying $\Phi',\Phi''\ge0$, and any $(t,x,v)\in\PP_T$, we have 
\begin{align*}
\LL\left(\Phi\circ\rho_t^2\right) &\ge \frac{1}{4}\aa\;\!\rr\;\!|\vv_d|\;\!\Phi' - C_\Lambda\aa\cc\;\!\rr^2\Phi'' - \Phi'\left(32\kappa\;\!\aa\sqrt{\frac{\aa}{\cc}}\;\!\rr^2 + C_\Lambda\cc +C_B\sqrt{\aa\cc}\;\!\rr \right), 
\end{align*}
where we used the abbreviations $\Phi'=\Phi'(\Phi\circ\rho_t^2)$ and $\Phi''=\Phi''(\Phi\circ\rho_t^2)$. 
\end{lemma}

\begin{proof}
Through direct computation, we find 
\begin{align*}
\partial_t\left(\Phi\circ\rho_t^2\right) &= -2\;\!\hh\;\!\vv_d\,\Phi'\big(\aa X_d^t-\bb V_d\big),\\
\nabla_x\left(\Phi\circ\rho_t^2\right) &= 2\;\!\Phi' \big(\kappa^2\aa\;\!x',\;\!\aa X_d^t-\bb V_d\;\!\big),\\
\nabla_v\left(\Phi\circ\rho_t^2\right) &= 2\;\!\Phi' \big(\cc\;\!v',\;\!\cc V_d-\bb X_d^t\;\!\big),\\
D_v^2\left(\Phi\circ\rho_t^2\right) &= 2\;\!\Phi'\cc I_d + 4\;\!\Phi''\big(\cc\;\! v',\;\!\cc V_d-\bb X_d^t\;\!\big)^{\otimes 2}. 
\end{align*}
It follows that 
\begin{align*}
\LL\left(\Phi\circ\rho_t^2\right) =&\,-2\hh\;\!\vv_d\;\!\Phi'(\aa X_d^t-\bb V_d) + 2\;\!\Phi'\big(\kappa^2\aa\;\!v'\cdot x'+v_d\;\!(\aa X_d^t-\bb V_d)\big) \\
&\,-4\;\!\Phi''A:\big(\cc\;\!v',\cc V_d-\bb X_d^t\;\!\big)^{\otimes2}
-2\;\!\Phi'\big(\cc A:I_d+B\cdot(\cc\;\!v',\cc V_d-\bb X_d^t) \big). 
\end{align*}
Under the conditions \eqref{vrs} and $\Phi', \Phi'' \ge 0$, part~\ref{vr2} of Lemma~\ref{hypodist2}, along with the boundedness of $A$ and $B$, implies that, for any $(t,x,v)\in\PP_T$, 
\begin{align*}
\LL\left(\Phi\circ\rho_t^2\right) \ge&\; \frac{1}{2}|\vv_d|(\aa\;\!\rr-4\hh|\bb V_d-\aa X_d^t|)\Phi'- 4\;\!\Phi''\Lambda|(\cc\;\!v',\cc V_d-\bb X_d^t)|^2\\
&-2\;\!\Phi'\big( \kappa^2\aa\;\!|x'||v'| + d\Lambda\;\!\cc +C_B\;\!\cc|v'| + C_B|\cc V_d-\bb X_d^t| \big). 
\end{align*}
Taking into account the ranges of $|x'|,|v'|,|X_d^t|,|V_d|$ specified by \eqref{range} from Lemma~\ref{hypodist1}, along with the assumptions in \eqref{abc}, one finds a constant $C_\Lambda\ge1$ depending only on $d$ and $\Lambda$ such that 
\begin{align*}
\frac{1}{2}|\vv_d|(\aa\;\!\rr-4\hh|\aa X_d^t-\bb V_d|)&\ge \frac{1-18\hh}{2}\;\!\aa\;\!\rr\;\!|\vv_d|,\\
4\Lambda|(\cc\;\!v',\cc V_d-\bb X_d)|^2&\le C_\Lambda\aa\cc\;\!\rr^2,\\
2\left(\kappa^2\aa\;\!|x'||v'| + d\Lambda\;\!\cc +C_B\;\!\cc|v'| + C_B|\cc V_d-\bb X_d| \right)&\le 32\kappa\;\!\aa\sqrt{\frac{\aa}{\cc}}\;\!\rr^2 +C_\Lambda \cc +C_B\sqrt{\aa\cc}\;\!\rr . 
\end{align*}
Gathering the three estimates above, together with the condition $\hh\in(0,\frac{1}{36}]$, leads to the stated result. 
\end{proof}

\subsubsection{Gradient estimates}
With the setup for barrier estimates in place, we proceed to derive gradient estimates for solutions to \eqref{KFP} on the incoming boundary $\Sigma_-$. 

\begin{proposition}\label{phase-prop}
Let $\zz=(\tt,\xx,\vv)=(0,\zero,\zero',\vv_d)\in\Sigma_-$ and $R\in(0,\min\{|\vv_d|,\langle v_0\rangle^{-2}\}]$, and let $f$ be a solution of \eqref{KFP} subject to \eqref{Elliptic} in $G_R(\zz)$ and $f=0$ on $\Sigma_-\cap G_R(\zz)$. Then, there is some constant $C>0$ depending only on $d$ and $\Lambda$ such that, for any $(t,x,v)\in G_R(\zz)$, we have 
\begin{align*}
|f(t,x,v)|\lesssim \left(R^{-3}|\vv_d||t| +R^{-3}|x-\xx|+R^{-1}|v-\vv|\right) \left(\|f\|_{L^\infty(G_R(\zz))} +R^2\|S\|_{L^\infty(G_R(\zz))}\right). 
\end{align*}
\end{proposition}

\begin{proof}
First, in the context of Lemma~\ref{barrier}, we choose $\rr,\kappa,\aa,\bb,\cc,\hh>0$ as follows, 
\begin{align}\label{abc-h}
\begin{aligned}
\rr^\frac{1}{3} \le \theta_0\min\big\{|\vv_d|,\langle v_0\rangle^{-2}\big\},\quad\kappa=1,\\
\aa:=\rr^{-\frac{2}{3}},\quad 
\bb:=\frac{1}{16},\quad 
\cc:=\frac{1}{4}\;\!\rr^\frac{2}{3},\quad
\hh:=\frac{1}{36}.
\end{aligned}
\end{align}
It is straightforward to verify that the conditions \eqref{vrs} and \eqref{abc} required by Lemma~\ref{barrier} are satisfied for any $\theta_0\in(0,\frac{1}{16}]$. Recall from Lemma~\ref{hypodist1} and \eqref{rho-t} that
\begin{align*}
\PP_T=\left\{\rho_0\le\rho(x-\vv\;\!t/36,v)\le3\rho_0\right\}\cap\OO_T {\quad\rm for\ \ } \rho_0=\rr^\frac{2}{3}. 
\end{align*}
We may assume that $r=c_1\rr^\frac{1}{3}$ and $R=c_2\rr^\frac{1}{3}$ for some constants $c_1,c_2>0$, taken so that $r\approx R\approx\rr^\frac{1}{3}$ and 
\begin{align*}
G_r(\zz) \subset \QQ:=\big\{-10\;\!\rr^\frac{2}{3}<t\le0,\ \rho_0\le\rho_t(x,v)\le3\rho_0 \big\}\cap\OO_T\subset G_R(\zz). 
\end{align*} 

Next, we consider the barrier function $\Phi\circ\rho_t^2$ defined on $\QQ$, where $\Phi:[\;\!\rho_0^2,\infty)\to\R_+$ is given by  
\begin{align*}
\Phi(\tau):=\rho_0^{-2}\tau-1.
\end{align*}
Applying Lemma~\ref{barrier} and noting that $\Phi'=\rho_0^{-2}=\rr^{-\frac{4}{3}}$ and $\Phi''=0$, along with the parameters $\rr,\kappa,\aa,\bb,\cc,\hh$ chosen in \eqref{abc-h}, yields that, for some constants $C_\Lambda,C_B>0$ with $C_B\lesssim\langle v_0\rangle^2$, 
\begin{align*}
\LL\left(\Phi\circ\rho_t^2\right) \ge \frac{1}{4}|\vv_d|\;\!\rr^{-1} - \rr^{-1}\big(C_\Lambda\;\!\rr^\frac{1}{3}+C_B\;\!\rr^\frac{2}{3} \big) {\quad\rm in\ }\QQ.
\end{align*}
By choosing the constant $\theta_0\in(0,\frac{1}{16}]$ sufficiently small so that 
\begin{align*}
1-16\;\!\theta_0\;\!C_\Lambda&\ge0,\\
1-16\;\!\theta_0\;\!C_B\langle v_0\rangle^{-2}&\ge0, 
\end{align*}
and observing that $\theta_0$ depends only on $d$ and $\Lambda$ owing to $C_B\lesssim\langle v_0\rangle^2$, we deduce 
\begin{align}\label{LL-Phi}
\left\{\begin{aligned}
\, &\LL\left(\Phi\circ\rho_t^2\right) \ge \frac{1}{8}|\vv_d|\;\!\rr^{-1} 
\gtrsim R^{-2}{\quad\rm in\ }\QQ,\\
\, &\;\Phi(9\rho_0^2)=8.
\end{aligned}\right. 
\end{align} 
Besides, by noticing that $\rho_{\tt}(\xx,\vv)=\rho(\xx,\vv)=\rho_0$ and $\Phi(\rho_0^2)=0$, we have
\begin{align*}
\Phi(\rho_t^2(x,v))&\le \Phi'\big[\left|\partial_t(\rho_t^2)\right||t|+\left|D_x(\rho_t^2)\right||x-\xx|+\left|D_v(\rho_t^2)\right||v-\vv|\big]\\
&\le 2\;\!\rr^{-\frac{4}{3}} \big[ (\aa|(x',X_d^t)|+\bb|V_d|)(\hh|\vv_d||t|+|x-\xx|) +(\cc|(v',V_d)|+\bb|X_d^t|)|v-\vv|\big]. 
\end{align*}
Given the ranges from \eqref{range} and choice of parameters $\kappa,\aa,\bb,\cc,\hh$ from \eqref{abc-h}, it follows that
\begin{align}\label{Phi-bound}
\Phi(\rho_t^2(x,v))\lesssim \rr^{-1}|\vv_d||t| +\rr^{-1}|x-\xx| +\rr^{-\frac{1}{3}}|v-\vv|  {\quad\rm in\ }\QQ. 
\end{align}

Let $M_f:=\|f\|_{L^\infty(G_R(\zz))} +  R^2\|S\|_{L^\infty(G_R(\zz))}$. Based on \eqref{KFP} and \eqref{LL-Phi}, we see that 
\begin{align*}
\LL\left(\pm f-C_0M_f\,\Phi\circ\rho_t^2\right)\le0 {\quad\rm in\ }G_r(\zz), 
\end{align*}
where the constant $C_0\ge1$ depends only on $d$ and $\Lambda$. By using the maximum principle (Lemma~\ref{max-su}) and \eqref{Phi-bound}, we obtain  
\begin{align*}
|f| \le C_0M_f\,\Phi\circ\rho_t^2 \lesssim M_fR^{-3}|\vv_d||t| +M_fR^{-3}|x-\xx| + M_fR^{-1}|v-\vv| {\quad\rm in\ }G_r(\zz). 
\end{align*}
This implies the stated result. 
\end{proof}

In view of the boundary flattening procedure described in Subsection~\ref{BFP}, Proposition~\ref{phase-prop} immediately implies part~\ref{rough-asmp1} of Theorem~\ref{rough-asmp}. Moreover, one obtains bounds for derivatives of arbitrary order for solutions to constant-coefficient equations via a standard induction argument, following, for instance, \cite[Corollary~3.3]{IM}, where the interior case was treated. Without loss of generality, we consider 
\begin{align*}
\LL_0:=\partial_t+v\cdot\nabla_x- \Delta_v. 
\end{align*}

\begin{corollary}\label{phase-grad}
Let $\zz\in\Sigma_-$, $R\in(0,\min\{|\vv_d|,1\}]$, and $l=(l_t,l_x,l_v)\in\mathbb{N}^{1+2d}$. Suppose that $f$ satisfies $\LL_0f=0$ in $G_R(\zz)$ and $f=0$ on $\Sigma_-\cap G_R(\zz)$. There is some constant $C_l>0$ depending only on $d,\lambda,\Lambda,|\vv|$ and $|l|$ such that 
\begin{align}\label{phase-grad-est}
\big|\partial_t^{l_t}D_x^{l_x}D_v^{l_v}f(\zz)\big|
\le C_l R^{-|l|_{\rm kin}}\|f\|_{L^\infty(G_R(\zz))} . 
\end{align}
\end{corollary}

\begin{proof}
The proof proceeds by nested induction first on $|l_x|$, then on $|l_v|$, and finally on $l_t$. For the base case $l_t=0$, Proposition~\ref{phase-prop} yields \eqref{phase-grad-est} whenever $|l_x|+|l_v|\le1$. Noting that the interior regularity then permits differentiation in $x$ and each $D_x^{l_x}f$ satisfies the same equation as $f$, we know that \eqref{phase-grad-est} extends inductively to every $l_x$ with $l_t=0$ and $|l_v|\le1$. Next assume that \eqref{phase-grad-est} holds for $l_t=0$, arbitrary $l_x$ and $1\le|l_v|\le n$. We observe that $D_x^{l_x}D_v^{l_v}f$ satisfies  
\begin{align*}
\LL_0\big(D_x^{l_x}D_v^{l_v}f\big)=-\sum\nolimits_{i=1}^d \,l_v\cdot\ee_i\, D_x^{l_x+\ee_i} D_v^{l_v-\ee_i}f,
\end{align*}
whose right-hand side is controlled by the induction hypothesis. Proposition~\ref{phase-prop} therefore gives the result for all $l_x$ and $|l_v|=n+1$ with $l_t=0$. Finally, the required bounds for all $l_t$ are obtained from the established estimates for derivatives of $f$ with respect to $x,v$ and the relation that $\partial_tf=\Delta_vf-v\cdot\nabla_xf$, thereby completing the proof. 
\end{proof}

\subsection{Infinite-order vanishing in the absence of influxes and sources}
Our next objective is to establish part~\ref{rough-asmp2} of Theorem~\ref{rough-asmp}, which asserts an infinite-order vanishing estimate with exponential-type decay towards the incoming boundary. 

\begin{lemma}\label{barrier-ss}
Let $(\xx,\vv)=(\zero,\zero',\vv_d)\in\Gamma_-$, and let $\rho_t$ and $\PP_T$ be defined in \eqref{rho-t} associated with $\rr,\kappa,\aa,\bb,\cc,\hh>0$. Consider a function $\Phi:\R_+\to\R_+$ such that $\Phi',\Phi''\ge0$. There exist some constants $\theta_0,C_0>0$ depending only on $d$ and $\Lambda$ such that, if we choose 
\begin{align}\label{r-abc}
\begin{aligned}
\rr\le\theta_0\min\big\{|\vv_d|^3,\langle v_0\rangle^{-6}\big\},\quad\kappa:=\frac{1}{64},\\
\aa:=|\vv_d|^2\;\!\rr^{-1},\quad 
\bb:=|\vv_d|,\quad 
\cc:=64\;\!\rr,\quad
\hh:=\frac{1}{36},
\end{aligned}
\end{align}
then these parameters satisfy the conditions \eqref{vrs} and \eqref{abc}; for any $(t,x,v)\in\PP_T$, we have
\begin{align*}
\LL\left(\Phi\circ\rho_t^2\right)\gtrsim C_0\;\!|\vv_d|^3\,\Phi'\circ\rho_t^2 -|\vv_d|^{-2}\rho_t^4\,\Phi''\circ\rho_t^2. 
\end{align*}
\end{lemma}

\begin{proof}
In view of Lemma~\ref{barrier}, we account for the parameters $\rr,\kappa,\aa,\bb,\cc,\hh$ as defined in \eqref{r-abc}, where the requirements \eqref{vrs} and \eqref{abc} are met for any $\theta_0\in(0,\frac{1}{16}]$. We choose $\theta_0$ to be sufficiently small so that 
\begin{align*}
1-1024\;\!C_\Lambda\;\!\theta_0&\ge0,\\
1-128\;\!\theta_0\;\!C_B\langle v_0\rangle^{-2}&\ge0, 
\end{align*}
where $C_\Lambda$ is given by Lemma~\ref{barrier}, and $\theta_0$ depends only on $d$ and $\Lambda$ as $C_B\lesssim\langle v_0\rangle^2$. Then, 
\begin{align*}
\frac{1}{8}\;\!\aa\;\!\rr|\vv_d| -32\kappa\;\!\aa\sqrt{\frac{\aa}{\cc}}\;\!\rr^2-C_\Lambda\cc
&=\frac{1}{8}\;\!|\vv_d|^3- \frac{1}{16}\;\!|\vv_d|^3 -64\;\!C_\Lambda\;\!\rr\\
&\ge\frac{1}{16}|\vv_d|^3 \big(1-1024\;\!C_\Lambda\;\!\theta_0\big)\ge0.  
\end{align*}
By noticing $|\vv_d|^2\ge\theta_0^{-\frac{2}{3}}\rr^\frac{2}{3}$ and $\rr^\frac{1}{3}\le\theta_0^\frac{1}{3}\langle v_0\rangle^{-2}$, we have 
\begin{align*}
\frac{1}{16}\;\!\aa\;\!\rr|\vv_d|-  C_B\sqrt{\aa\cc}\;\!\rr 
&=\frac{1}{16}|\vv_d|^3-8\;\!C_B\;\!|\vv_d|\;\!\rr \\
&\ge\frac{1}{16}\;\!\theta_0^\frac{1}{3}\;\!|\vv_d|\;\!\rr^\frac{2}{3}\big(\theta_0^{-1}-128\;\!C_B\langle v_0\rangle^{-2}\big)\ge0. 
\end{align*}
Combining the above two estimates with Lemma~\ref{barrier}, we obtain, for any $(t,x,v)\in\PP_T$, 
\begin{align*}
\LL\left(\Phi\circ\rho_t^2\right) &\ge \frac{1}{16}\;\!\aa\;\!\rr\;\!|\vv_d|\,\Phi'\circ\rho_t^2 -C_\Lambda\;\!\aa\cc\;\!\rr^2\,\Phi''\circ\rho_t^2 \\
&= \frac{1}{16}|\vv_d|^3\,\Phi'\circ\rho_t^2 -64\;\!C_\Lambda|\vv_d|^2\;\!\rr^2\,\Phi''\circ\rho_t^2. 
\end{align*}
Recalling that $\rho_t^2\ge\rho_0^2=|\vv_d|^2\;\!\rr$, we establish the result as claimed. 
\end{proof}

The following technical lemma, which will be employed in the proof of Proposition~\ref{barrier-exp} below, characterizes the decay property for solutions to a relevant ordinary differential equation. 

\begin{lemma}\label{ode}
Let $\Theta>0$ be a constant. Consider the function $\phi:\R_+\to\R$ defined by 
\begin{align*}
\phi'(\tau)=\exp\left(-\frac{\Theta}{\tau}\right), \quad
\phi(\tau)=\int_0^\tau\phi'(\tilde{\tau}){\;\!\rm d}\tilde{\tau}. 
\end{align*}
For fixed $\tau_0>0$, define the function $\Phi:[\tau_0,\infty)\to\R_+$ as  
\begin{align*}
\Phi(\tau):=\frac{\phi(\tau)-\phi(\tau_0)}{\phi(9\tau_0)-\phi(\tau_0)}. 
\end{align*} 
Then, $\Phi$ satisfies the ordinary differential equation
\begin{align*}
\tau^2\Phi''(\tau)-\Theta\,\Phi'(\tau)=0, 
\end{align*} 
and for any $\tau\in[0,4\tau_0]$, 
\begin{align*}
\Phi(\tau) \le \Phi(4\tau_0) \le \left(1+\frac{\Theta}{\tau_0}\right) \exp\left(-\frac{\Theta}{8\tau_0}\right). 
\end{align*}
\end{lemma}

\begin{proof}
We can make the simplification so as to assume that $\Theta=1$; otherwise, consider the rescaled function $\tau\mapsto\phi(\Theta\tau)$. By definition of $\phi$ and $\Phi$, it is straightforward to check that 
\begin{align*}
&\phi(0)=\phi'(0)=\Phi(\tau_0)=\Phi'(\tau_0)=0,\\
&\phi,\ \phi',\ \phi''\ge0 {\quad\rm on\quad}\R_+. 
\end{align*}
Moreover, both $\phi$ and $\Phi$ solve the aforementioned ordinary differential equation with $\Theta=1$. Next, we observe that for $\tau\in[0,9\tau_0]$, 
\begin{align*}
(\tau^2\phi'(\tau))'=\tau^2\phi''(\tau) +2\tau\phi'(\tau)
=(1+2\tau)\phi'(\tau)\le(1+18\tau_0)\phi'(\tau). 
\end{align*}
Integrating both sides from $\tau_0$ to $9\tau_0$ yields
\begin{align*}
(9\tau_0)^2\phi'(9\tau_0)-\tau_0^2\phi'(\tau_0) \le(1+18\tau_0)\left[\phi(9\tau_0)-\phi(\tau_0)\right], 
\end{align*}
which, by the monotonicity of $\phi'$, implies 
\begin{align*}
\phi(9\tau_0)-\phi(\tau_0) \ge \frac{80\tau_0^2}{1+18\tau_0}\phi'(9\tau_0). 
\end{align*}
Additionally, by the mean value theorem and the monotonicity of $\phi'$ again, we obtain
\begin{align*}
0\le \phi(4\tau_0)-\phi(\tau_0)\le 3\tau_0\phi'(4\tau_0).
\end{align*}
Combining the above two estimates, we find
\begin{align*}
\frac{\phi(4\tau_0)-\phi(\tau_0)}{\phi(9\tau_0)-\phi(\tau_0)}
\le \frac{3(1+18\tau_0)}{80\tau_0} \frac{\phi'(4\tau_0)}{\phi'(9\tau_0)}
= \frac{3(1+18\tau_0)}{80\tau_0} \exp\left(-\frac{5}{36\tau_0}\right). 
\end{align*} 
This is sufficient to conclude the asserted upper bound for $\Phi(4\tau_0)$. 
\end{proof}

\begin{proposition}\label{barrier-exp}
Let $\zz=(\tt,\xx,\vv)=(0,\zero,\zero',\vv_d)\in\Sigma_-$, and let $f$ satisfy \eqref{KFP} subject to \eqref{Elliptic} with $S=0$ in $G_1(\zz)$ and $f=0$ on $\Sigma_-\cap G_1(\zz)$. There exists some constant $c_\Lambda\in(0,1)$ depending only on $d$ and $\Lambda$ such that for any $0\le-x_d\le|\vv_d|^3$, we have
\begin{align}\label{ssexp}
|f(0,\zero',x_d,\zero',\vv_d)|\le\|f\|_{L^\infty(G_1(\zz))}\exp\left(1-\frac{c_\Lambda\min\{|\vv_d|^3,1\}}{\langle v_0\rangle^6|x_d|}\right).  
\end{align}
\end{proposition}

\begin{proof}
Let $\rr,\kappa,\aa,\bb,\cc,\hh,\theta_0,C_0>0$ be given by Lemma~\ref{barrier-ss}, the function $\rho_t(x,v)$ be defined in \eqref{rho-t}. For some $0<\rr\le\theta_0\min\{|v_d|^3,\langle v_0\rangle^{-6}\}$ to be chosen, we set 
\begin{align*}
\Theta:=C_0|\vv_d|^5 {\quad\rm and\quad }\rho_0^2:=|\vv_d|^2\;\!\rr. 
\end{align*}
Recalling the region $\PP_T$ defined in \eqref{rho-t}, we can assume that, for some constants $c_1,c_2>0$, 
\begin{align*}
\QQ_1\subset \{-10\;\!|\vv_d|\;\!\rr^{-1}<t\le0,\ \rho_0\le\rho_t\le2\rho_0 \}\cap\OO_T \subset
\{-10\;\!|\vv_d|\;\!\rr^{-1}<t\le0\}\cap\PP_T\subset\QQ_2, 
\end{align*}
where $\QQ_i:=(-c_i|\vv_d|^{-1}\rr,0]\times\left(B_{c_i\rr}(\xx)\cap\mathbb{H}_-^d\right)\times B_{c_i|\vv_d|}(\vv)$ for $i=1,2$. 
Define the function 
\begin{align*}
\Phi(\tau):=\frac{\phi(\tau)-\phi(\rho_0^2)}{\phi(9\rho_0^2)-\phi(\rho_0^2)}, 
\end{align*} 
where $\phi=\phi(\tau)$ is given in Lemma~\ref{ode}, ensuring that
\begin{align*}
\tau^2\Phi''(\tau)-\Theta\,\Phi'(\tau)=0. 
\end{align*}
The function $\Phi$ adheres to the upper bound established in Lemma~\ref{ode}, meaning that, for any $\rho_t\le2\rho_0$, we have 
\begin{align*}
\Phi(\rho_t^2)\le \left(1+\frac{\Theta}{\rho_0^2}\right) \exp\left(-\frac{\Theta}{8\rho_0^2}\right)
\le \exp\left(1-\frac{c_0|\vv_d|^3}{\rr}\right),
\end{align*} 
for some $c_0>0$ depending only on $d$ and $\Lambda$. Leveraging Lemma~\ref{barrier-ss} and the ordinary differential equation satisfied by $\Phi$, we obtain, for any $(t,x,v)\in\PP_T$, 
\begin{align*}
\LL\left(\Phi\circ\rho_t^2\right)\gtrsim |\vv_d|^{-2}\left(\Theta\,\Phi'\circ\rho_t^2 -\rho_t^4\,\Phi''\circ\rho_t^2\right)= 0. 
\end{align*}
In addition, we have $\Phi(9\rho_0^2)=1$. Applying the maximum principle (Lemma~\ref{max-su}) to the functions $\pm f-\|f\|_{L^\infty(\QQ_2)}\Phi\circ\rho_t^2$ in the region $(-10|\vv_d|^{-1}\rr,0]\cap\PP_T$ yields that 
\begin{align*}
|f(t,x,v)|\le \|f\|_{L^\infty(\QQ_2)}\;\!\Phi(4\rho_0^2)
\le \|f\|_{L^\infty(\QQ_2)}\exp\left(1-\frac{c_0|\vv_d|^3}{\rr}\right) {\quad\rm in\ }\QQ_1.  
\end{align*}
It remains to verify that \eqref{ssexp} follows. Indeed, for any $|x_d|\le c_1\theta_0\min\{|v_d|^3,\langle v_0\rangle^{-6}\}$, we set $\rr=c_1^{-1}|x_d|$ and then arrive at \eqref{ssexp} for $c_\Lambda\le c_0c_1$. If $c_1\theta_0\min\{|v_d|^3,\langle v_0\rangle^{-6}\}\le|x_d|\le|\vv_d|^3$, then \eqref{ssexp} holds trivially for $c_\Lambda\le \theta_0c_1$. We thus conclude the proof by picking $c_\Lambda:=c_1\min\{c_0,\theta_0\}$. 
\end{proof}

The conclusion of part~\ref{rough-asmp2} of Theorem~\ref{rough-asmp} then follows as an immediate consequence of the preceding proposition together with the boundary flattening procedure presented in Subsection~\ref{BFP}.

\subsection{H\"older regularity near the grazing boundary}
We now aim to prove Theorem~\ref{rough-holder}. In contrast to the $\lambda$-independent estimates at the incoming boundary $\Sigma_-$ involved in the preceding two subsections, the transport effect diminishes as one approaches the grazing set $\Sigma_0$. Geometrically, the condition~\eqref{vrs} in part~\ref{vr2} of Lemma~\ref{hypodist2} is no longer satisfied in this regime; one may compare Figure~\ref{img-in} with Figure~\ref{img-g}. To derive a certain amount of regularity near $\Sigma_0$, it becomes essential to fully exploit the hypoelliptic structure of the operator $\LL=\partial_t+v\cdot\nabla_x- A:D^2_v-B\cdot\nabla_v$, wherein the interplay between the transport term and the ellipticity of the matrix $A$ has a meaningful role. 

Note that, near the grazing boundary, it is not necessary to introduce the time-dependent function $\rho_t$ involving the spatial translation $X_d^t$ set up in \S~\ref{evolu}. Indeed, one may always assume $|\vv_d|\le 1$ in a neighborhood of the grazing boundary $\Gamma_0=\{x_d=0,v_d=0\}$. Let us proceed accordingly. 

\begin{lemma}\label{barrier-g}
Let $(\xx,\vv)=(\zero,\zero',\vv_d)\in\Gamma_-$, and let $\rho$ and $\PP$ be defined in Lemma~\ref{hypodist1} associated with $\rr,\kappa,\aa,\bb,\cc>0$. There exist some constants $C_\Lambda,m\ge1$ and $\theta_0\in(0,\frac{1}{32}]$ depending only on $d,\lambda,\Lambda$ such that, if we choose 
\begin{align}\label{abc-g}
\begin{aligned}
\rr\le\min\big\{|\vv_d|^3,\langle v_0\rangle^{-6}\big\},\quad\kappa:=\frac{\sqrt{\theta_0}}{256},\\
\aa:=\rr^{-\frac{2}{3}},\quad \bb:=\theta_0,\quad 
\cc:=2\theta_0\;\!\rr^{\frac{2}{3}}, \quad \hh:=\theta_0, 
\end{aligned}
\end{align}
then these parameters satisfy the conditions \eqref{vr} and \eqref{abc}; furthermore, we have 
\begin{align*}
\LL\left(\varphi(\rho-\hh t)\right) \gtrsim \rr^{-\frac{2}{3}} {\quad\rm in\ \,}(-\infty,0]\times\PP, 
\end{align*}
where the function $\varphi:[\rho_0,\infty)\to\R_+$, with $\rho_0:=\rr^\frac{2}{3}$, is defined by   
\begin{align*}
\varphi(\rho):=\frac{\rho^{-m}-\rho_0^{-m}}{(3\rho_0)^{-m}-\rho_0^{-m}}. 
\end{align*}
\end{lemma}

\begin{proof}
Within the setting of \eqref{abc-g}, the conditions \eqref{vr} and \eqref{abc} are satisfied for any $\theta_0\in(0,\frac{1}{32}]$. According to Lemma~\ref{hypodist1}, we consider the phase domain 
\begin{align*}
\PP=\{\rho_0\le\rho\le3\rho_0\}\cap\OO {\quad\rm for\ \ }
\rho_0=\sqrt{\aa}\;\!\rr=\rr^\frac{2}{3}.
\end{align*}
A direct computation yields that 
\begin{align*}
\partial_t\left(\varphi(\rho-\hh t)\right) &= -\hh\;\!\varphi',\\
\nabla_x\left(\varphi(\rho-\hh t)\right) &= \varphi'\rho^{-1} \left(\kappa^2\aa\;\!x',\aa X_d-\bb V_d\right),\\
\nabla_v\left(\varphi(\rho-\hh t)\right) &= \varphi'\rho^{-1} \left(\cc\;\!v',\,\cc V_d-\bb X_d\right),\\
D_v^2\left(\varphi(\rho-\hh t)\right)&= \varphi'\rho^{-1}\cc I_d +\left(\varphi''\rho^{-2}-\varphi'\rho^{-3}\right)(\cc\;\!v',\cc V_d-\bb X_d)^{\otimes 2}, 
\end{align*} 
where we abbreviated $\varphi'=\varphi'(\rho-\hh t)$ and $\varphi''=\varphi''(\rho-\hh t)$. It follows that 
\begin{align*}
v\cdot\nabla_x\left(\varphi(\rho-\hh t)\right) &= \varphi'\rho^{-1} \left(v',v_d\right)\cdot\left(\kappa^2\aa\;\!x',\aa X_d-\bb V_d\right)\\
&= \varphi'\rho^{-1} \left(\kappa^2\aa\;\!x'\cdot v' +V_d\left(\aa X_d-\bb V_d\right) +\eta_d\left(\aa X_d-\bb V_d\right) \right). 
\end{align*}
Observing from its definition that $\varphi'\ge0$ and $\rho\;\!\varphi''=-(m+1)\varphi'\le0$, we derive  
\begin{align*}
A:D_v^2\varphi &= \varphi'\rho^{-1}\cc\;\!A:I_d +\left(\varphi''\rho^{-2}-\varphi'\rho^{-3}\right)A:(\cc\;\!v',\cc V_d-\bb X_d)^{\otimes 2}. \\
&\le d\Lambda\;\!\cc\;\!\varphi'\rho^{-1}  -(m+2)\;\!\lambda\;\!\varphi'\rho^{-3}|\cc V_d-\bb X_d|^2.
\end{align*}
By collecting the above two computations, we obtain 
\begin{align}\label{LLL}
\begin{aligned}
\LL\left(\varphi(\rho-\hh t)\right)\ge&\; \varphi'\rho^{-1}\left(Q(X_d,V_d)+\eta_d\;\!(\aa X_d-\bb V_d)\right) -\hh\;\!\varphi'\\
&-\varphi'\rho^{-1}\left(\kappa^2\aa\;\!|x'||v'| + d\Lambda\;\!\cc + C_B\;\!\cc|v'| + C_B|\cc V_d-\bb X_d| \right). 
\end{aligned}
\end{align}
Here, with the abbreviation
\begin{align*}
p_m:=(m+2)\;\!\lambda\;\!\rho^{-2}, 
\end{align*}
the quadratic form $Q:\R\times\R\to\R$ above is defined by
\begin{align*}
Q(X_d,V_d):=p_m\bb^2X_d^2 +\left(\aa-2p_m\bb\cc\right)X_d\;\!V_d +\left(p_m\cc^2-\bb\right)V_d^2. 
\end{align*}
In view of the ranges of $|x'|,|v'|,|X_d|,|V_d|$ specified by \eqref{range}, along with \eqref{abc}, we have 
\begin{align}\label{LLLo}
\kappa^2\aa\;\!|x'||v'| + d\Lambda\;\!\cc +C_B\;\!\cc|v'| + C_B|\cc V_d-\bb X_d|
\le 32\kappa\;\!\aa\sqrt{\frac{\aa}{\cc}}\;\!\rr^2 + C_\Lambda\cc+ C_B\sqrt{\aa\cc}\;\!\rr, 
\end{align}
for some constant $C_\Lambda>0$ depending only on $d$ and $\Lambda$. Armed with \eqref{vr}, we apply part~\ref{vr1} of Lemma~\ref{hypodist2} to \eqref{LLL}, combined with \eqref{LLLo}, to deduce that for any $(t,x,v)\in(-\infty,0]\times\PP$, 
\begin{align}\label{psilemma0}
\LL\left(\varphi(\rho-\hh t)\right) \ge \varphi'\rho^{-1} \bigg(Q_{\min}+\frac{\aa\;\!\rr\;\!|\vv_d|}{4} -32\kappa\;\!\aa\sqrt{\frac{\aa}{\cc}}\;\!\rr^2 -C_\Lambda\cc -C_B\sqrt{\aa\cc}\;\!\rr\bigg)-\hh\;\!\varphi' . 
\end{align} 
Here we are set to establish a lower bound for $Q_{\min}:=\min_{(x,v)\in\PP}Q(X_d,V_d)$. Assuming that
\begin{align}\label{pm}
p_m=(m+2)\;\!\lambda\;\!\rho^{-2}\ge\frac{\aa}{\bb\cc}, 
\end{align}
we know from $\sqrt{\aa\cc}\ge8\bb$ that $p_m\cc^2-\bb>0$. The discriminant $\varDelta_Q$ of $Q(\cdot,\cdot)$ is given by 
\begin{align*}
\varDelta_Q:\!&= (2p_m\bb\cc-\aa)^2 -4p_m\bb^2(p_m\cc^2-\bb)\\
&=-4\;\!p_m\aa\bb\cc+4\;\!p_m\bb^3+\aa^2. 
\end{align*}
Under the assumptions \eqref{pm} and $\sqrt{\aa\cc}\ge8\bb$, it is negative; to be more precise, 
\begin{align*}
-p_m^{-1}\varDelta_Q = 4\;\!\aa\bb\cc-4\;\!\bb^3-p_m^{-1}\aa^2 \ge 2\;\!\aa\bb\cc. 
\end{align*}
With the constraint $|X_d|\ge\rr$ from \eqref{range}, the quantity $Q_{\min}$ is thus positive and can be bounded from below as 
\begin{align*}
Q_{\min} \ge \frac{-\varDelta_Q\;\!X_d^2}{4\left(p_m\cc^2-\bb\right)}
\ge \frac{\aa\bb\;\!\rr^2}{2\;\!\cc}. 
\end{align*}
Combining this with \eqref{psilemma0}, we obtain, for any $(t,x,v)\in(-\infty,0]\times\PP$, 
\begin{align}\label{psilemma1}
\LL\left(\varphi(\rho-\hh t)\right) \ge&\; \varphi'\rho^{-1} \bigg(\frac{\aa\bb\;\!\rr^2}{2\;\!\cc}+\frac{\aa\;\!\rr\;\!|\vv_d|}{4} - 32\kappa\;\!\aa\sqrt{\frac{\aa}{\cc}}\;\!\rr^2 -C_\Lambda\;\!\cc -C_B\sqrt{\aa\cc}\;\!\rr\bigg) -\hh\;\!\varphi'. 
\end{align}

Under the setting of \eqref{abc-g}, let $\theta_0>0$ be sufficiently small so that
\begin{align*}
32\;\!C_\Lambda\;\!\theta_0+16\;\!C_B\sqrt{2\theta_0}\;\!\langle v_0\rangle ^{-2}\le1.
\end{align*} 
Then, we have
\begin{align}\label{psi-au1}
\begin{aligned}
\frac{\aa\bb\;\!\rr^2}{4\;\!\cc} -32\kappa\;\!\aa\sqrt{\frac{\aa}{\cc}}\;\!\rr^2&=\frac{1}{8}\;\!\rr^\frac{2}{3} \left(1-\frac{1}{\sqrt{2}} \right)\ge0,\\
\frac{\aa\bb\;\!\rr^2}{4\;\!\cc} -C_\Lambda\;\!\cc -C_B\sqrt{\aa\cc}\;\!\rr
&= \frac{1}{8}\;\!\rr^\frac{2}{3} \left(1-16C_\Lambda\theta_0 -8C_B\sqrt{2\theta_0}\;\!\rr^\frac{1}{3} \right)
\ge \frac{1}{16}\;\!\rr^\frac{2}{3}. 
\end{aligned}
\end{align} 
Since $\rho\le3\rho_0=3\;\!\rr^{\frac{2}{3}}$ in $\PP$, we can verify \eqref{pm} by selecting $m=5\lambda^{-1}\theta_0^{-2}$; indeed, we have
\begin{align*}
p_m = (m+2)\;\!\lambda\;\!\rho^{-2}
\ge \frac{m+2}{9}\;\!\lambda\;\!\rr^{-\frac{4}{3}} 
>\frac{1}{2}\;\!\theta_0^{-2}\;\!\rr^{-\frac{4}{3}}=\frac{\aa}{\bb\cc}. 
\end{align*} 
Besides, it is straightforward to check that
\begin{align}\label{psi-au2}
\frac{\aa\;\!\rr\;\!|\vv_d|}{4\;\!\rho}\ge\frac{\aa\;\!\rr^\frac{4}{3}}{12\;\!\rho_0}=\frac{1}{12}>\hh. 
\end{align}
We thus deduce from \eqref{psilemma1}, \eqref{psi-au1} and \eqref{psi-au2} that 
\begin{align*}
\LL\left(\varphi(\rho-\hh t)\right) \ge \frac{1}{16}\;\!\varphi'\rho^{-1}\;\! \rr^\frac{2}{3}\ge \frac{1}{48}\;\!\varphi' {\quad\rm in\ \,}(-\infty,0]\times\PP. 
\end{align*}
Noting from its definition that 
\begin{align*}
\varphi'=m\;\!\rho^{-m-1}(\rho_0^{-m}-(3\rho_0)^{-m})\approx\rho_0^{-1}=\rr^{-\frac{2}{3}}, 
\end{align*}
we hence conclude the proof. 
\end{proof}

\begin{proposition}\label{osc-g}
Let $r_0\in(0,\langle v_0\rangle^{-2}]$, and let $f$ be a solution of \eqref{KFP} subject to \eqref{Elliptic} in $G_{r_0}$ and $f=f_b$ on $\Sigma_-$. There are some constants $\delta,c\in(0,1)$ and $C>0$ depending only on $d,\lambda,\Lambda$ such that 
\begin{align*}
{\;\!\rm osc\;\!}_{G_{cr_0}} f\le \delta{\;\!\rm osc\;\!}_{G_{r_0}} f +{\;\!\rm osc\;\!}_{G_{r_0}} f_b +C\|S\|_{L^\infty(G_{r_0})}. 
\end{align*}
\end{proposition}

\begin{proof}
Let $\rr\in(0,\langle v_0\rangle^{-6}]$, and $\kappa,\aa,\bb,\cc,\hh>0$ be specified in \eqref{abc-g} of Lemma~\ref{barrier-g}. Consider a point $(\xx,\vv)\in\Gamma_-$ with $\xx=\zero$, $\vv=(\zero',\vv_d)$ and $\vv_d=-\rr^\frac{1}{3}$. We examine the function $\varphi(\rho-\hh t)$ as introduced in Lemma~\ref{barrier-g}. Recall the region of interest is $\PP=\{\rho_0\le\rho\le3\rho_0\}\cap\OO$ for $\rho_0=\rr^\frac{2}{3}$. 
\begin{figure}
\includegraphics[width=9.35cm]{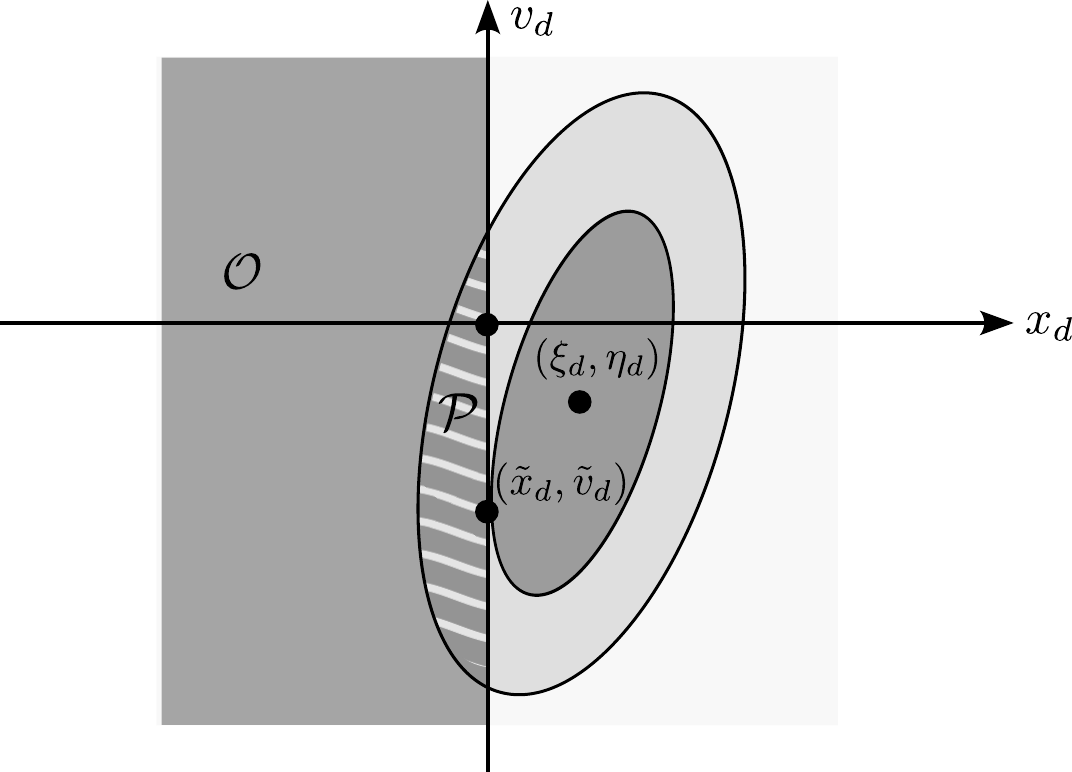}\\\vspace{-0.15cm}
\caption{The region $\PP$ that contains the grazing set.}
\label{img-g}
\end{figure}
We now check that 
\begin{align}\label{PP0}
(x_0,v_0)=(\zero,\zero)\in\{\rho_0\le\rho\le2\rho_0\}\cap\OO\subset\PP. 
\end{align}
One may refer to Figure~\ref{img-g}. Indeed, it suffices to verify that 
\begin{align*}
\rho^2(x_0,v_0)=\aa\;\!\xi_d^2-2\bb\;\!\xi_d\;\!\eta_d+\cc\;\!\eta_d^2
\le4\rho_0^2=4\;\!\rr^\frac{4}{3}. 
\end{align*}
We know from \eqref{xieta} of Lemma~\ref{hypodist1} and \eqref{abc-g} of Lemma~\ref{barrier-g} that 
\begin{align*}
&\rr\le \xi_d\le \frac{4}{3}\;\!\rr\\
&|\eta_d|=|\vv_d|-\frac{\bb\;\!\xi_d}{\cc}
=\rr^\frac{1}{3}-\frac{\xi_d}{2}\;\!\rr^{-\frac{2}{3}}\le \frac{1}{2}\;\!\rr^\frac{1}{3}. 
\end{align*}
From our choice of $\aa,\bb,\cc$ in \eqref{abc-g} with $\theta_0\in(0,\frac{1}{32}]$, it turns out that
\begin{align*}
\aa\;\!\xi_d^2-2\bb\;\!\xi_d\;\!\eta_d+\cc\;\!\eta_d^2
\le \frac{16}{9}\;\!\rr^{\frac{4}{3}} + \frac{4}{3}\;\!\theta_0\;\!\rr^\frac{4}{3} +\frac{1}{2}\;\!\theta_0\;\!\rr^\frac{4}{3}
\le 4\;\!\rr^{\frac{4}{3}},
\end{align*}
which implies \eqref{PP0} as claimed. Furthermore, we can assume that
\begin{align*}
G_r &\subset \QQ:=\big\{-10\;\!\theta_0^{-1}\;\!\rr^\frac{2}{3}<t\le0,\ \rho_0\le\rho- \hh t\le 2\rho_0 \big\}\cap\OO_T \\
&\subset \big(-10\;\!\theta_0^{-1}\;\!\rr^\frac{2}{3},0\big]\times\PP\subset G_R, 
\end{align*}
with $r:=cR$ and $\rr:=c'R^3$ for some constants $c,c'>0$ so that $r\approx R\approx\rr^\frac{1}{3}\le\langle v_0\rangle^{-2}$. 

Let set $\delta:=\varphi(2\rho_0)\in(0,1)$. It follows from Lemma~\ref{barrier-g} that, for some constant $c_0>0$, 
\begin{align*}
\left\{\begin{aligned}
\,&\LL\left(\varphi(\rho-\hh t)\right)\ge c_0R^{-2}  & &{\rm in\ \;} \QQ,\\
\,&\;\varphi(\rho-\hh t)\le \delta & &{\rm in\ \;} G_r, \\
\,&\;\varphi(\rho-\hh t)\ge1 & &{\rm on\ \;} \partial\QQ\backslash\{x_d=0, {\;\rm or\ }t=0\}. 
\end{aligned}\right. 
\end{align*} 
Define $M_*:=\sup_{G_R}f+c_0^{-1}R^2\;\!\|S\|_{L^\infty(G_R)}$. Applying the maximum principle (Lemma~\ref{max-su}) to the function $f-\sup_{G_r\cap\Sigma_-}f-M_*\varphi(\rho-\hh t)$, we conclude that 
\begin{align*}
f\le \sup\nolimits_{G_r\cap\Sigma_-}f + \delta\sup\nolimits_{G_R}f + C\|S\|_{L^\infty(G_R)}  {\quad\rm in\ }G_r. 
\end{align*}
Replacing $f$ by $-f$, we have 
\begin{align*}
-f\le -\inf\nolimits_{G_r\cap\Sigma_-}f - \delta\inf\nolimits_{G_R}f + C\|S\|_{L^\infty(G_R)} {\quad\rm in\ }G_r. 
\end{align*}
Summing the two estimates above leads to the desired result. 
\end{proof}

Theorem~\ref{rough-holder}, asserting the H\"older regularity of the solution on the grazing boundary, is a standard consequence of the oscillation decay presented in Proposition~\ref{osc-g}. 

\section{Regular coefficients case}\label{sec-regular}
This section is concerned with establishing Theorem~\ref{regular-Schauder}, which addresses higher-order regularity away from the grazing boundary, and Theorem~\ref{regular-optimal}, which characterizes the optimal H\"older regularity near the grazing boundary. By virtue of the boundary flattening reduction from Lemma~\ref{flatten}, our analysis is carried out in the same half-space domain $\OO_T=(-\infty,0]\times\OO=(-\infty,0]\times\mathbb{H}_-^d\times\R^d$ as in the preceding section. 

\subsection{Gradient estimates away from the grazing boundary}
We focus here on constant-coefficient equations. For concreteness, and without loss of generality, we work with the operator $\LL_0=\partial_t+v\cdot\nabla_x- \Delta_v$. Let us invoke the interior gradient estimates established in \cite[Proposition~3.1 and Corollary~3.3]{IM}, which remain valid on the outgoing part of the boundary. In combination with Corollary~\ref{phase-grad}, this yields the following higher-order gradient estimates for solutions to constant-coefficient equations away from the grazing boundary, for which a sketch of the proof is provided below. 

\begin{lemma}\label{phase-grad2}
Let $\zz\in\OO_T$, $R\in(0,1]$, and $l=(l_t,l_x,l_v)\in\mathbb{N}^{1+2d}$. Suppose that $f$ satisfies $\LL_0f=0$ in $G_R(\zz)$. If $G_R(\zz)\cap(\Sigma_0\cup\Sigma_-)=\emptyset$, then there is some constant $C_l>0$ depending only on $d,\lambda,\Lambda,|\zz|$ and $|l|$ such that
\begin{align*}
|\partial_t^{l_t}D_x^{l_x}D_v^{l_v}f|(\zz)
\le C_l R^{-|l|_{\rm kin}}\|f\|_{L^\infty(G_R(\zz))} . 
\end{align*}
If, in addition,  $f|_{\Sigma_-}=0$, then the same estimate holds for any $\zz\in\OO_T$ such that $G_R(\zz)\cap\Sigma_0=\emptyset$. 
\end{lemma}

\begin{proof}
In the case where $G_R(\zz)\cap(\Sigma_0\cup\Sigma_-)=\emptyset$ (see Figure~\ref{img-out} for an illustration when $\zz\in\Sigma_+$), the local estimates are independent of the boundary data imposed on $\Sigma_-$. The corresponding gradient estimates for $f$ can be established by an adaptation of Bernstein's method, along the same lines as the approach in \cite[Proposition 3.1]{IM}. Higher order gradient estimates are then obtained via a standard induction argument; see \cite[Corollary~3.3]{IM}. 

If additionally $f|_{\Sigma_-}=0$, then Corollary~\ref{phase-grad} implies desired estimates on $\Sigma_-$ in the case where $G_R(\zz)\cap\Sigma_0=\emptyset$. The boundary and interior estimates can be assembled to obtain the desired derivative bounds; see for instance Lemma~\ref{patching}. 
\end{proof}

\begin{figure}
\includegraphics[width=8.55cm]{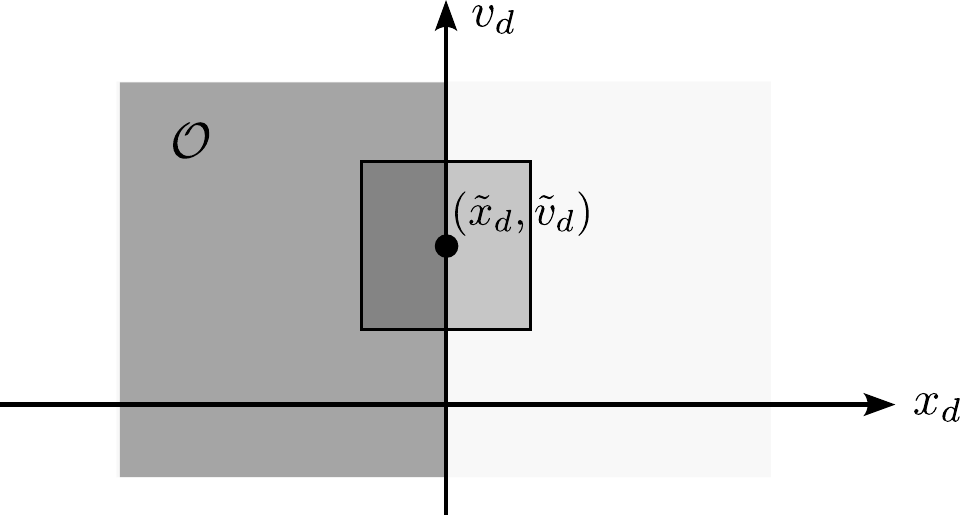}\\\vspace{-0.15cm}
\caption{Solutions follow the same pattern on the outgoing boundary as in the interior.}
\label{img-out}
\end{figure}

\subsection{Schauder estimates away from the grazing boundary}
In what follows, we incorporate a perturbative argument to complete the derivation of the Schauder estimates asserted in part~\ref{regular-Schauder-h} of Theorem~\ref{regular-Schauder}. Passing from gradient estimates to Schauder estimates via perturbation arguments has become standard, relying only on the maximum principle, since works such as \cite{Safonov,Caf1}. For the sake of completeness, we present a proof that, in essence, employs a kinetic adaptation of Safonov's approach. The argument applies equally to the derivation of higher-order interior Schauder estimates. 

Before turning to the proof, we record the following fundamental iteration lemma from \cite[Lemma~1.1]{GG} for subsequent use in our argument. 
\begin{lemma}\label{technicallemma}
Let $n(r)$ be a nonnegative bounded function defined for $0\le r_0\le r\le r_1$. Suppose that, for any $r_0\le r<R\le r_1$,  
\begin{align*}
n(r)\le \epsilon n(s)+M(s-r)^{-\iota},  
\end{align*}
for some constants $\iota,M\ge 0$ and $\epsilon\in[0,1)$. Then, for any $r_0\le r<R\le r_1$, we have
\begin{align*}
n(r)\le CM(R-r)^{-\iota}, 
\end{align*}
where $C>0$ is a constant depending only on $\epsilon$ and $\iota$. 
\end{lemma}

Part~\ref{regular-Schauder-h} of Theorem~\ref{regular-Schauder} follows from the estimates furnished by the proposition stated below, together with the boundary flattening procedure described in Subsection~\ref{BFP}. 

\begin{proposition}\label{schauder-bdy}
Let $z_0\in\R^{1+2d}$ satisfy $G_{2r_0}(z_0)\cap\Sigma_0=\emptyset$ for some $r_0\in(0,1]$. Let $f$ be a solution of \eqref{KFP} subject to \eqref{Elliptic} in $G_{2r_0}(z_0)$ with $f|_{\Sigma_-}=0$, and $A,B,S\in\Ck^{k+\alpha}(G_{2r_0}(z_0))$ for some $k\in\mathbb{N}$ and $\alpha\in(0,1)$. There exists some constant $C_k>0$ depending only on $d,\lambda,\Lambda,|z_0|$, $k,\alpha$ and $\Ck^{k+\alpha}(G_{2r_0}(z_0))$-norms of $A,B$ such that 
\begin{align*}
r_0^{k+2+\alpha}\;\![f]_{\Ck^{k+2+\alpha}(G_{r_0}(z_0))}
\le C_k\big(\|f\|_{L^\infty(G_{2r_0}(z_0))} + \|S\|_{\Ck^{k+\alpha}(G_{2r_0}(z_0))}\big).
\end{align*}
\end{proposition}

\begin{proof}
Let $0<r<R\le R_0\le1$, and let $\zz\in G_r(z_0)$ and $\theta\in(0,1]$ such that $G_\theta(\zz)\subset G_R(z_0)$. Note that this inclusion implies $\theta\lesssim R-r$. Denote by $T_\zz[f]$ the polynomial expansion of $f$ around $\zz$, truncated at kinetic degree $k+2$. Without loss of generality, we may assume that $A(\zz)=I_d$. Consider the function $h$ solving
\begin{align*}
\left\{\begin{aligned}
\, &\LL_0h=0 {\quad\rm in\ } G_\theta(\zz),\\
\, &\; h=f-T_\zz[f] {\quad\rm on\ } \partial_{\rm kin} G_\theta(\zz). \\
\end{aligned}\right. 
\end{align*} 
One may refer to \cite[Lemma 2.7]{Zhu} for the solvability of the above boundary value problem. By applying the maximum principle from Lemma~\ref{max-k}, we obtain
\begin{align*}
\|h\|_{L^\infty(G_\theta(\zz))}= \|h\|_{L^\infty(\partial_{\rm kin}G_\theta(\zz))}
\le \|f-T_\zz[f]\|_{L^\infty(G_\theta(\zz))}
\le \theta^{k+2+\alpha}[f]_{\Ck^{k+2+\alpha}(G_R(z_0))}.
\end{align*} 
Combining this with the gradient estimates from Lemma~\ref{phase-grad2}, we know that for any left-invariant differential operator $D^l$ with $l\in\mathbb{N}^{1+2d}$, and any $\varepsilon\in(0,\frac{1}{2}]$, 
\begin{align}\label{bound-l}
\|D^lh\|_{L^\infty(G_{\varepsilon\theta}(\zz))}\lesssim \theta^{-|l|_{\rm kin}} \|h\|_{L^\infty(G_\theta(\zz))} 
\lesssim \theta^{k+2+\alpha-|l|_{\rm kin}}[f]_{\Ck^{k+2+\alpha}(G_R(z_0))}. 
\end{align}
By the polynomial expansion of $h$ around $\zz$ up to kinetic degree $k+2$ and the mean value theorem, for any $z\in G_{\varepsilon\theta}$, there exists some $z'\in G_{\varepsilon\theta}$ such that 
\begin{align*}
h(\zz\circ z)-T_\zz[h](z)=\sum\nolimits_{|l|=k+2} c_l\big[D^lf(\zz\circ z')-D^lf(\zz)\big] z^l, 
\end{align*}
where the constant $c_l>0$ depends only on $l$. This implies that 
\begin{align*}
\|h-T_\zz[h]\|_{L^\infty(G_{\varepsilon\theta}(\zz))} \lesssim 
\sum\nolimits_{|l|_{\rm kin}=k+3,\;\!k+4,\;\!k+5} (\varepsilon\theta)^{|l|_{\rm kin}}\|D^lh\|_{L^\infty(G_{\varepsilon\theta}(\zz))}. 
\end{align*}
Applying \eqref{bound-l} to this estimate yields 
\begin{align}\label{hhg}
\begin{aligned}
\|h-T_\zz[h]\|_{L^\infty(G_{\varepsilon\theta}(\zz))}
&\le C\varepsilon^{k+3}\theta^{k+2+\alpha}[f]_{\Ck^{k+2+\alpha}(G_R(z_0))}\\
&\le \frac{1}{4}(\varepsilon\theta)^{k+2+\alpha}[f]_{\Ck^{k+2+\alpha}(G_R(z_0))},
\end{aligned}
\end{align}
where the second inequality follows from choosing $\varepsilon>0$ sufficiently small so that $C\varepsilon^{1-\alpha}\le\frac{1}{4}$. 

Define $S_f:=(\LL_0-\LL)f+S$, so that $\LL_0f=S_f$. Recalling \S~\ref{d-poly}, it is straightforward to verify that the function $f-T_\zz[f]-h$ satisfies 
\begin{align*}
\left\{\begin{aligned}
\, &\LL_0(f-T_\zz[f]-h)= S_f-T^k_\zz[S_f] {\quad\rm in\ } G_\theta(\zz),\\
\, &\; f-T_\zz[f]-h=0 {\quad\rm on\ } \partial_{\rm kin}G_\theta(\zz), \\
\end{aligned}\right. 
\end{align*}
for $T^k_\zz[S_f]$ denoting the polynomial expansion of $S_f$ around $\zz$ of kinetic degree $k$. The right-hand side of the above equation satisfies 
\begin{align*}
\|S_f-T^k_\zz[S_f]\|_{L^\infty(G_\theta(\zz))}
\le C\theta^{k+\alpha} [S_f]_{\Ck^{k+\alpha}(G_R(z_0))}. 
\end{align*}
It follows from the maximum principle in Lemma~\ref{max-k} that 
\begin{align}\label{ggh}
\theta^{-k-2-\alpha}\|f-T_\zz[f]-h\|_{L^\infty(G_\theta(\zz))}
\le C [S_f]_{\Ck^{k+\alpha}(G_R(z_0))}
\end{align}
By the H\"older condition on $\LL$ and the interpolation of kinetic H\"older spaces (see for instance \cite[Proposition~2.10]{IS-Schauder}), we have
\begin{align*}
[S_f]_{\Ck^{k+\alpha}(G_R(z_0))}
\le (CR_0^\alpha+\varepsilon')[f]_{\Ck^{k+2+\alpha}(G_R(z_0))} + C_{\varepsilon'}\|f\|_{L^\infty(G_R(z_0))} + C[S]_{\Ck^{k+\alpha}(G_R(z_0))}.  
\end{align*} 
Here, the constant $C_{\varepsilon'}>0$ depending on $\varepsilon'$ arises from the interpolation, and we are free to choose $R_0,\varepsilon'\in(0,1]$ sufficiently small so that \eqref{ggh} implies 
\begin{align}\label{gghh}
\begin{aligned}
\theta^{-k-2-\alpha}\|f-T_\zz[f]-h\|_{L^\infty(G_\theta(\zz))}
\le&\ \frac{1}{4}\,[f]_{\Ck^{k+2+\alpha}(G_R(z_0))} \\
&\,+ C\|f\|_{L^\infty(G_R(z_0))} + C[S]_{\Ck^{k+\alpha}(G_R(z_0))}.
\end{aligned}
\end{align}

Next, \eqref{hhg} and \eqref{gghh} can be incorporated into the following triangle inequality  
\begin{align*}
\|f-T_\zz[f+h]\|_{L^\infty(G_{\varepsilon\theta}(\zz))}
\le \|h-T_\zz[h]\|_{L^\infty(G_{\varepsilon\theta}(\zz))} + \|f-T_\zz[f]-h\|_{L^\infty(G_{\varepsilon\theta}(\zz))}, 
\end{align*} 
where $\zz\in G_r(z_0)$ is arbitrary, subject to $G_\theta(\zz)\subset G_R(z_0)$, requiring $\theta\lesssim R-r$. For the complementary case $R-r\lesssim\theta\lesssim R$, the quantity $\theta^{-k-2-\alpha}\|f-T_\zz[f+h]\|_{L^\infty(G_\theta(\zz))}$ admits a trivial bound in terms of the $\Ck^{k+2}(G_\theta(z_0))$-norms of $f$ and $h$, where $h$ can, in turn, be estimated by $f$ in a manner analogous to \eqref{bound-l}. More precisely, invoking the interpolation yields that for $R-r\lesssim\theta\lesssim R$, 
\begin{align*}
\theta^{-k-2-\alpha}\|f-T_\zz[f+h]\|_{L^\infty(G_\theta(\zz))}
&\le C\theta^{-k-2-\alpha}\|f\|_{L^\infty(G_\theta(z_0))} + C\theta^{-\alpha}[f]_{\Ck^{k+2}(G_\theta(z_0))}\\
&\le C(R-r)^{-k-2-\alpha}\|f\|_{L^\infty(G_R(z_0))} + \frac{1}{2}\,[f]_{\Ck^{k+2+\alpha}(G_R(z_0))}. 
\end{align*} 
From the analysis of these two situations, we deduce that for any $0<r<R\le R_0$, 
\begin{align*}
[f]_{\Ck^{k+2+\alpha}(G_r(z_0))}
\le&\;\frac{1}{2}\,[f]_{\Ck^{k+2+\alpha}(G_R(z_0))} + C [S]_{\Ck^{k+\alpha}(G_R(z_0))}\\
&\, + C(R-r)^{-k-2-\alpha} \|f\|_{L^\infty(G_R(z_0))}.
\end{align*}
By virtue of Lemma~\ref{technicallemma}, we derive the claimed estimate. 
\end{proof}

\subsection{Cordes-Nirenberg estimates away from the grazing boundary}
We next state the Cordes-Nirenberg estimate for \eqref{KFP}, which applies to the equation with leading coefficients of merely small oscillation at small scales, less regular than those typically considered in the classical Schauder theory. For completeness, we include the proof below, which proceeds along almost the same lines as the derivation of the Schauder estimate in Proposition~\ref{schauder-bdy}. 

Part~\ref{regular-Schauder-CN} of Theorem~\ref{regular-Schauder} is implied by the following proposition. 

\begin{proposition}\label{CN-bdy}
Let $z_0\in\R^{1+2d}$ satisfy $G_{2r_0}(z_0)\cap\Sigma_0=\emptyset$ for some $r_0\in(0,1]$, and let $\alpha\in(0,1)$. There exist some $\varrho,R_0\in(0,1]$ depending only on $d,\lambda,\Lambda,\alpha$ such that if $f$ is a solution of \eqref{KFP} subject to \eqref{Elliptic} in $G_{2r_0}(z_0)$ with $f|_{\Sigma_-}=0$, with the leading coefficients satisfying
\begin{align}\label{small}
\omega(R_0):=\sup\nolimits_{\{z,\zz:\,z\in G_{R_0}(\zz)\subset G_{2r_0}(z_0)\}}|A(z)-A(\zz)|\le \varrho, 
\end{align}
then for some $C_\omega>0$ depending only on $d,\lambda,\Lambda,|z_0|,\alpha$, we have 
\begin{align*}
r_0^{1+\alpha}\;\![f]_{\Ck^{1+\alpha}(G_{r_0}(z_0))}
\le C_\omega\big( \|f\|_{L^\infty(G_{2r_0}(z_0))} + \|S\|_{L^\infty(G_{2r_0}(z_0))} \big).
\end{align*}
\end{proposition}

\begin{proof}
Let $0<r<R\le R_0\le1$, and let $\zz\in G_r(z_0)$ and $\theta\in(0,1]$ such that $G_\theta(\zz)\subset G_R(z_0)$. Note that this inclusion implies $\theta\lesssim R-r$. Consider the function $h$ solving
\begin{align*}
\left\{\begin{aligned}
\, &\LL_0h=0 {\quad\rm in\ } G_\theta(\zz),\\
\, &\; h=f-L_\zz[f] {\quad\rm on\ } \partial_{\rm kin}G_\theta(\zz), \\
\end{aligned}\right.
\end{align*} 
where we assume $A(\zz)=I_d$ and denote by $L_\zz[f]$ the polynomial expansion of $f$ around $\zz$ of kinetic degree $1$. By applying the maximum principle from Lemma~\ref{max-k}, we obtain
\begin{align*}
\|h\|_{L^\infty(G_\theta(\zz))}= \|h\|_{L^\infty(\partial_{\rm kin}G_\theta(\zz))}
\le \|f-L_\zz[f]\|_{L^\infty(G_\theta(\zz))}
\le \theta^{1+\alpha}[f]_{\Ck^{1+\alpha}(G_R(z_0))}.
\end{align*}
By Lemma~\ref{phase-grad2}, we know that for any $D^l$ with $l\in\mathbb{N}^{1+2d}$, and any $\varepsilon\in(0,\frac{1}{2}]$, 
\begin{align}\label{h-cn}
\|D^lh\|_{L^\infty(G_{\varepsilon\theta}(\zz))}\lesssim \theta^{-|l|_{\rm kin}} \|h\|_{L^\infty(G_\theta(\zz))} 
\lesssim \theta^{1+\alpha-|l|_{\rm kin}}[f]_{\Ck^{1+\alpha}(G_R(z_0))}. 
\end{align}
The polynomial expansion of $h$ around $\zz$ of kinetic degree $1$, together with the mean value theorem, shows that 
\begin{align*}
\|h-L_\zz[h]\|_{L^\infty(G_{\varepsilon\theta}(\zz))}\le 
\sum\nolimits_{|l|_{\rm kin}=2,\;\!3,\;\!4} (\varepsilon\theta)^{|l|_{\rm kin}}\|D^lh\|_{L^\infty(G_{\varepsilon\theta}(\zz))}. 
\end{align*}
Combining this with \eqref{h-cn}, and taking $\varepsilon>0$ to be small so that $C\varepsilon^{1-\alpha}\le\frac{1}{4}$, we have 
\begin{align}\label{hhg-cn}
\begin{aligned}
\|h-L_\zz[h]\|_{L^\infty(G_{\varepsilon\theta}(\zz))}
\le C\varepsilon^2\theta^{1+\alpha}[f]_{\Ck^{1+\alpha}(G_R(z_0))}
\le \frac{1}{4}(\varepsilon\theta)^{1+\alpha}[f]_{\Ck^{1+\alpha}(G_R(z_0))}.
\end{aligned}
\end{align}
Additionally, the function $f-L_\zz[f]-h$ satisfies 
\begin{align*}
\left\{\begin{aligned}
\, &\LL(f-L_\zz[f]-h)=S_h:= S-B\cdot\nabla_vf(\zz) +(\LL_0-\LL)h {\quad\rm in\ } G_\theta(\zz),\\
\, &\; f-L_\zz[f]-h=0 {\quad\rm on\ } \partial_{\rm kin}G_\theta(\zz).\\
\end{aligned}\right. 
\end{align*}
By \eqref{h-cn}, the right-hand side of the above equation satisfies
\begin{align*}
\|S_h\|_{L^\infty(G_{\varepsilon\theta}(\zz))}
\lesssim \|S\|_{L^\infty(G_R(z_0))} 
+\|f\|_{\Ck^1(G_R(z_0))} + \omega(R_0)\,\theta^{\alpha-1}[f]_{\Ck^{1+\alpha}(G_R(z_0))}. 
\end{align*}
By the maximum principle in Lemma~\ref{max-k} and the interpolation, we obtain 
\begin{align*}
\begin{aligned}
\|f-L_\zz[f]-h\|_{L^\infty(G_\theta(\zz))}
\lesssim&\ \omega(R_0)\;\!\theta^{1+\alpha}[f]_{\Ck^{1+\alpha}(G_R(z_0))}\\ 
&\,+\theta^2\|S\|_{L^\infty(G_R(z_0))} + \theta^2\|f\|_{L^\infty(G_R(z_0))}. 
\end{aligned}
\end{align*}
Applying this and \eqref{hhg-cn} in following the triangle inequality  
\begin{align*}
\|f-L_\zz[f+h]\|_{L^\infty(G_{\varepsilon\theta}(\zz))}\le \|h-L_\zz[h]\|_{L^\infty(G_{\varepsilon\theta}(\zz))} + \|f-L_\zz[f]-h\|_{L^\infty(G_{\varepsilon\theta}(\zz))}, 
\end{align*} 
we deduce that for any $\zz\in G_r(z_0)$ such that $\theta\lesssim R-r$, 
\begin{align*}
\theta^{-1-\alpha}\|f-L_\zz[f+h]\|_{L^\infty(G_{\varepsilon\theta}(\zz))}
\le \Big(\frac{1}{4}+C\omega(R_0)\Big)[f]_{\Ck^{1+\alpha}(G_R(z_0))} + C\|S\|_{L^\infty(G_R(z_0))}. 
\end{align*} 
For $R-r\lesssim\theta\lesssim R$, by using the interpolation, we have 
\begin{align*}
\theta^{-1-\alpha}\|f-L_\zz[f+h]\|_{L^\infty(G_{\varepsilon\theta}(\zz))}
\le C(R-r)^{-1-\alpha}\|f\|_{L^\infty(G_R(z_0))} +\frac{1}{4}\,[f]_{\Ck^{1+\alpha}(G_R(z_0))}.
\end{align*} 
The combination of the above two estimates yields
\begin{align*}
[f]_{\Ck^{1+\alpha}(G_r(z_0))}
\le &\;\Big(\frac{1}{4}+C\omega(R_0)\Big)[f]_{\Ck^{1+\alpha}(G_R(z_0))} + C\|S\|_{L^\infty(G_R(z_0))}\\ 
&\ + C(R-r)^{-1-\alpha}\|f\|_{L^\infty(G_R(z_0))}. 
\end{align*} 
Choose $R_0\in(0,1)$ so that the small-oscillation assumption \eqref{small} is satisfied for some small $\varrho\in(0,1)$ ensuring $C\varrho\le\frac{1}{4}$. Then, we have
\begin{align*}
[f]_{\Ck^{1+\alpha}(G_r(z_0))}
\le \frac{1}{2}\,[f]_{\Ck^{1+\alpha}(G_R(z_0))} + C\|S\|_{L^\infty(G_R(z_0))}
+ C(R-r)^{-1-\alpha} \|f\|_{L^\infty(G_R(z_0))}. 
\end{align*}
By applying Lemma~\ref{technicallemma}, we deduce the desired estimate. 
\end{proof}

\subsection{Optimal H\"older regularity near the grazing boundary}
This subsection is dedicated to the proof of Theorem~\ref{regular-optimal}, the optimal H\"older regularity for solutions of \eqref{KFP} with H\"older continuous leading coefficients, up to the grazing set $\Sigma_0$. According to Remark~\ref{remark-reduction}, we may now localize the analysis near the origin of $\R^{1+2d}$. 

To initiate the argument, we invoke the construction of a stationary solution exhibiting sharp H\"older regularity at $\Sigma_0$ from \cite[Lemma~2.1]{GJW}; see also \cite[Claim 3.7]{HJV} and \cite[Appendix~A]{Zhu} for similar results. 

\begin{lemma}\label{phi-lemma}
There is a nonnegative function $\psi=\psi(x,v)$ with $(x,v)=(x',x_d,v',v_d)\in\mathbb{H}_-^d\times\R^d$ that is independent of the variables $(x',v')\in\R^{d-1}\times\R^{d-1}$ and satisfies 
\begin{align*}
\left\{\begin{aligned}
\ &v\cdot\nabla_x\;\!\psi=0 &{\rm for\ \;}&x_d\le0,\ v_d\in\R,\\
\ &\;\! \psi=0 &{\rm for\ \;}&x_d=0,\ v_d\le0, \\
\ &\;\! \psi\approx (-x_d)^\frac{1}{6} &{\rm in\ \;}&\RR_0\,:=\left\{\;\!0\le (c_*v_d)^3\le -x_d \;\!\right\}\cap(B_1(0)\times B_1(0)), \\
\ &\;\! \psi\approx \sqrt{v_d} &{\rm in\ \;}&\RR_+:=\left\{\;\!0\le-x_d\le (c_*v_d)^3\;\!\right\}\cap(B_1(0)\times B_1(0)),\\
\ &\;\! \psi\approx \sqrt{|v_d|}\, e^{-v_d^3/(9x_d)} &{\rm in\ \;}&\RR_-:=\left\{\;\!0<-x_d\le -(c_*v_d)^3\;\!\right\}\cap(B_1(0)\times B_1(0)),
\end{aligned}\right. 
\end{align*} 
for some the constant $c_*\in(0,1)$. See Figure~\ref{img-graze} for a depiction of the regions $\RR_0$ and $\RR_\pm$. 
\end{lemma}

\begin{figure}
\includegraphics[width=8.45cm]{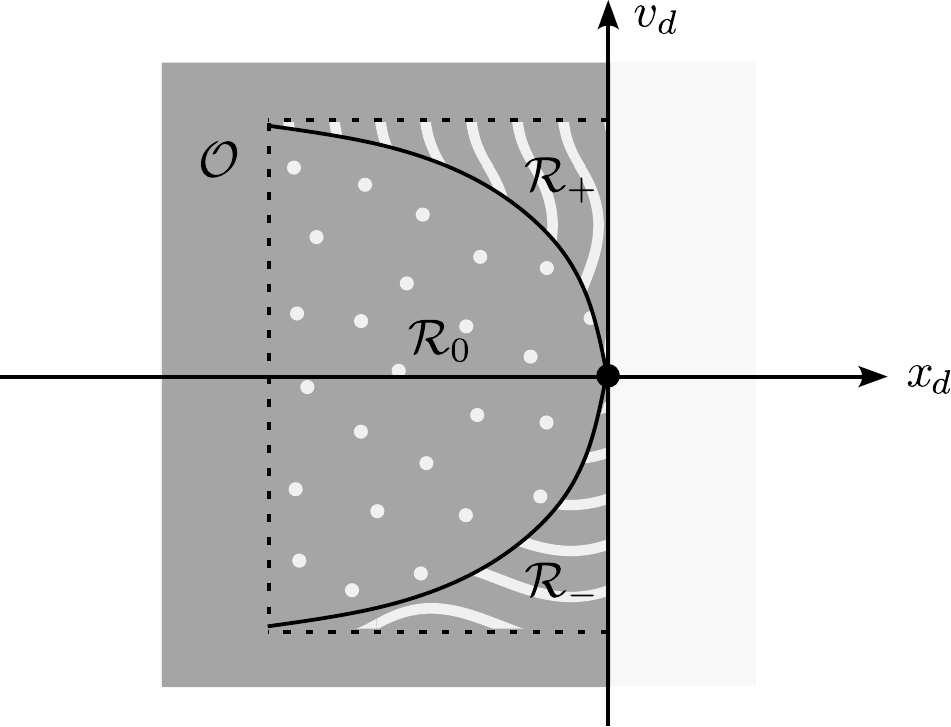}\\\vspace{-0.15cm}
\caption{The regions $\RR_0$ and $\RR_\pm$ near the grazing set where the solution exhibits distinct asymptotic behaviours.}
\label{img-graze}
\end{figure}

\begin{proof}
It suffices to consider the case $d=1$, in which we write $x=x_d\in\R_-$ and $v=v_d\in\R$. We construct a special solution $\psi$ of $v\partial_x\psi=0$ on $\R_-\times\R$ based on the ansatz
\begin{align*}
\psi(x,v)=(-x)^\frac{1}{6}\,\Upsilon(\tau) {\quad\rm with\ \;}\tau=-\frac{v^3}{9x},
\end{align*}
under which the profile function $\Upsilon:\R\to\R_+$ must satisfy
\begin{align}\label{ode-Tricomi}
\tau\Upsilon''(\tau)+\bigg(\frac{2}{3}-\tau\bigg)\Upsilon'(\tau)+\frac{1}{6}\Upsilon(\tau)=0. 
\end{align}
According to \cite[Sections~7.10]{Olver}, we can choose 
\begin{align*}
\Upsilon(\tau)=
\left\{\begin{aligned}
&\ U\!\left(-\frac{1}{6},\frac{2}{3},\tau\right) &{\rm for\ \;}&\tau\ge0,\\
&\ \frac{e^\tau}{6}\,U\!\left(\,\frac{5}{6},\frac{2}{3},-\tau\right)  &{\rm for\ \;}& \tau\le0, 
\end{aligned}\right. 
\end{align*}
where $U(\cdot,\cdot,\cdot)$ denotes Tricomi's (confluent hypergeometric) function. Its definition and asymptotic behaviours can also be found in \cite[Sections~7.10]{Olver}. While $\Upsilon(\tau)$ satisfies \eqref{ode-Tricomi} for $\tau\neq0$ by the definition of $U$, the validity of \eqref{ode-Tricomi} across $\tau=0$ is confirmed by the explicit expressions of $U(\cdot,\cdot,0)$ for the value and derivative of $\Upsilon$ at that point: 
\begin{align*}
\Upsilon(0)&=\lim_{\tau\,\to\,0^\pm}\Upsilon(\tau)=\frac{\varGamma(1/3)}{\varGamma(1/6)},\\
\Upsilon'(0)&=\lim_{\tau\,\to\,0^\pm}\Upsilon(\tau)=\frac{\varGamma(-2/3)}{6\,\varGamma(1/6)},
\end{align*}
where $\varGamma(\cdot)$ demotes the Gamma function. The identity $\frac{2}{3}\;\!\varGamma(-2/3)=\varGamma(1/3)$ thus yields \eqref{ode-Tricomi} at $\tau=0$. Furthermore, $\Upsilon>0$ on $\R$, and 
\begin{align*}
\Upsilon(\tau)=
\left\{\begin{aligned}
&\ \tau^\frac{1}{6}\;\!\big(1+O\big(\tau^{-1}\big)\big) &{\rm as\ \;}&\tau\to\infty,\\
&\ \frac{e^\tau}{6}\,|\tau|^\frac{1}{6}\;\!\big(1+O\big(\tau^{-1}\big)\big) &{\rm as\ \;}& \tau\to-\infty.
\end{aligned}\right. 
\end{align*}
This in turn implies that for any $(x,v)\in\R_-\times\R$ and any $a\neq0$,  
\begin{align*}
\psi(x,v)=
\left\{\begin{aligned}
&\ |x|^\frac{1}{6}\;\!\Upsilon\big(-(9a)^{-1}\big) &{\rm as\ \;}&xv^{-3}\to a,\\
&\ \,v^\frac{1}{2}\;\!\big(9^{-\frac{1}{6}}+O\big(xv^{-3}\big)\big) &{\rm as\ \;}&xv^{-3}\to0^-,\\
&\ \,\frac{1}{6}\,|v|^\frac{1}{2}\,e^{-\frac{v^3}{9x}}\;\!\big(9^{-\frac{1}{6}}+O\big(xv^{-3}\big)\big) &{\rm as\ \;}& xv^{-3}\to0^+.
\end{aligned}\right. 
\end{align*}
By choosing $c_*\in(0,1)$ sufficiently small, we ensure that for any $|xv^{-3}|\le c_*^3$, 
\begin{align*}
1/2\le 9^{-\frac{1}{6}}+O\big(xv^{-3}\big)\le 1, 
\end{align*} 
from which the asserted asymptotics of $\psi$ follows. 
\end{proof}

Building upon the precise characterization of the boundary behaviours of the special stationary solution $\psi$ from the preceding lemma, we are able to construct a suitable barrier function for the constant-coefficient operators, exemplified by $\LL_0=\partial_t+v\cdot\nabla_x-\Delta_v$, subject to the zero influx boundary condition. 

\begin{lemma}\label{phi-barrier}
Let the function $\psi(x,v)$ and the regions $\RR_0,\RR_\pm$ with the constant $c_*\in(0,1)$ be given in Lemma~\ref{phi-lemma}. There are some constants $r_0,c_0>0$ such that for the function 
\begin{align*}
\Psi(t,x,v):=\psi(x,v)-2v_d-|v_d|^2-t, 
\end{align*}
we have 
\begin{align*}
\left\{\begin{aligned}
\ &\LL_0\Psi=1 &\ &{\rm in\ \;\,} \{x_d\le 0\},\\
\ &\;\!\Psi\lesssim \|z\|^\frac{1}{2} &\ &\;\!{\rm in\ \;}\ \,\!G_{r_0},\\
\ &\;\!\Psi\ge0 &\ &{\rm on\ \;} \{x_d=0,\ v_d\le0\}, \\
\ &\;\!\Psi\ge c_0 &\ &{\rm on\ \;} \{t=-r_0^2,{\ \rm or\ }x_d=-r_0^3,{\ \rm or\ }|v_d|=r_0\}\cap\overline{G}_{r_0}.  
\end{aligned}\right. 
\end{align*} 
\end{lemma}

\begin{proof}
Given the function $\psi$ constructed in Lemma~\ref{phi-lemma}, it is straightforward to see that   
\begin{align*}
\left\{\begin{aligned}
\, &\LL_0\Psi=1 &\ &{\rm in\ \;}\, \{x_d\le 0\},\\
\, &\,\Psi=0 &\ &{\rm on\ \;} \{x_d=v_d=0\}. 
\end{aligned}\right. 
\end{align*} 
The upper bound of $\Psi$ follows directly from its definition together with the result of Lemma~\ref{phi-lemma}. It thus suffices to establish the lower bound of $\Psi$ on the parabolic boundary of $G_{r_0}$ for some $r_0>0$. By Lemma~\ref{phi-lemma}, there is some constant $c\in(0,1)$ such that 
\begin{align*}
\psi(x,v)\ge 2c\;\!|x_d|^\frac{1}{6}\ge 2c\;\!|c_*v_d|^\frac{1}{2} & {\quad\rm in\ \;}\RR_0, \\
\psi(x,v)\ge 2c\;\!|v_d|^\frac{1}{2}& {\quad\rm in\ \;}\RR_+,  
\end{align*} 
where $\RR_0$ and $\RR_+$ are defined in Lemma~\ref{phi-lemma} in terms of $c_*\in(0,1)$. 
Choosing $r_0\in(0,1)$ to be sufficiently small, we know that for any $r\in(0,r_0]$,
\begin{align*}
2c\sqrt{c_*r}-2r-r^2\ge c\sqrt{c_*r} {\quad\rm and\quad}
2r-r^2\ge  r, 
\end{align*} 
which in particular implies that 
\begin{align*}
\psi(x,v)-2v_d-|v_d|^2\ge 0 {\quad\rm in\ \;}\overline{G}_{r_0}. 
\end{align*} 
Recalling the definition of $\Psi$, we deduce that 
\begin{align*}
\left\{\begin{aligned}
\,&\;\!\Psi\ge r_0^2 &\ &{\rm on\ \;} \{t=-r_0^2\}\cap\overline{G}_{r_0}, \\
\,&\;\!\Psi\ge c\sqrt{c_*r_0} &\ &{\rm on\ \;} \{x_d=-r_0^3,{\ \rm or\ }v_d=r_0\}\cap\overline{G}_{r_0},  \\
\,&\;\!\Psi\ge r_0 &\ &{\rm on\ \;} \{v_d=-r_0\}\cap\overline{G}_{r_0}. 
\end{aligned}\right. 
\end{align*} 
This establishes the claimed lower bound. 
\end{proof}

This barrier construction identifies and saturates the optimal H\"older exponent. On the one hand, it incorporates an explicit solution whose behaviour exhibits precisely this exponent, ruling out any possible improvement. On the other hand, the barrier argument yields the sharp H\"older regularity to \eqref{KFP}, matching the exponent realized by the barrier. 

Let us begin by applying this barrier construction to \eqref{KFP} with constant coefficients. 

\begin{lemma}\label{phi-holder}
Let $f$ be a solution to $\LL_0f=S$ in $G_r$ with $f|_{\Sigma_-}=0$. Then, we have 
\begin{align*}
r^{1/2}\;\![f]_{\Ck^{1/2}(G_{r/2})} \lesssim \|f\|_{L^\infty(G_r)}+\|S\|_{L^\infty(G_r)}. 
\end{align*} 
\end{lemma}

\begin{proof}
Let the constants $r_0,c_0\in(0,1)$ and the function $\Psi$ be given by Lemma~\ref{phi-barrier}. By scaling, it suffices to prove the estimate with $r=r_0$. Let $M_*:=c_0^{-1}\|f\|_{L^\infty(G_{r_0})}+\|S\|_{L^\infty(G_{r_0})}$. In the light of the properties of $\Psi$ established in Lemma~\ref{phi-barrier}, we have
\begin{align*}
\left\{\begin{aligned}
\, &\LL_0(f-M_*\Psi)\le0 &\ &{\rm in\ \;\,} \{x_d\le 0\}\cap\overline{G}_{r_0},\\
\, &\;\!f-M_*\Psi\le0 &\ &{\rm on\ \;} \{x_d=0,\ v_d\le0\}\cap\overline{G}_{r_0}, \\
\, &\;\!f-M_*\Psi\le0 &\ &{\rm on\ \;} \{t=-r_0^2,{\ \rm or\ }x_d=-r_0^3,{\ \rm or\ }|v_d|=r_0\}\cap\overline{G}_{r_0}. 
\end{aligned}\right. 
\end{align*} 
Applying the maximum principle, Lemma~\ref{max-k}, yields that for any $z\in G_{r_0}$, 
\begin{align*}
f(z)&\le \big(c_0^{-1}\|f\|_{L^\infty(G_{r_0})}+\|S\|_{L^\infty(G_{r_0})}\big)\,\Psi(z)\\
&\lesssim \big(\|f\|_{L^\infty(G_{r_0})}+\|S\|_{L^\infty(G_{r_0})}\big)\,\|z\|^\frac{1}{2}, 
\end{align*}
where the last inequality above is also given by Lemma~\ref{phi-barrier}. 
\end{proof}

We are in a position to prove the optimal H\"older regularity for solutions of \eqref{KFP} with various coefficients, based on a perturbative argument in the same spirit as the Schauder estimates. Theorem~\ref{regular-optimal} follows as a direct consequence of the following proposition, together with the boundary flattening procedure from Subsection~\ref{BFP}. 

\begin{proposition}\label{sharp-holder}
Let $z_0\in\Sigma_0$ and $r_0,\alpha\in(0,1)$. Let $f$ be a solution of \eqref{KFP} subject to \eqref{Elliptic} in $G_{2r_0}(z_0)$ with $f|_{\Sigma_-}=0$ and $A\in\Ck^\alpha(G_{2r_0}(z_0))$. There exists some constant $C>0$ depending only on $d,\lambda,\Lambda,|z_0|,\alpha$ and $[A]_{\Ck^\alpha(G_{2r}(z_0))}$ such that
\begin{align*}
r_0^{1/2}\;\![f]_{\Ck^{1/2}(G_{r_0})} \le C\big( \|f\|_{L^\infty(G_r)} + \|S\|_{L^\infty(G_{2r_0})} \big). 
\end{align*} 
\end{proposition}

\begin{proof}
It suffices to derive the H\"older estimate at the origin $z_0=0$. Let us abbreviate the cylinder $G^k:=G_{2^{-k}}(z_0)=G_{2^{-k}}$ with $k\ge0$ so that $G^k\subset G_{R/2}$. For $l>k\ge0$, we consider the domain
\begin{align*}
G^k_l:=G^k\cap\{|x_d|\ge2^{-3l},\,|v_d|\ge2^{-l}\}, 
\end{align*}
which is away from the grazing boundary $\Sigma_0$. Let the function $f_k$ satisfy 
\begin{align*}
\left\{\begin{aligned}
\,&\LL_0f_k=0 & &{\rm in\ } G^k,\\
\,&\;f_k=f & &{\rm on\ } \partial_{\rm kin}G^k. \\
\end{aligned}\right. 
\end{align*}
Here we assume, without loss of generality, that $A(z_0)=I_d$. By Lemma~\ref{phi-holder}, we have 
\begin{align}\label{hk1}
\|f_k\|_{\Ck^{1/2}(G^k)} 
\lesssim 2^{k/2}\|f\|_{L^\infty(G_R)} . 
\end{align} 
The gradient estimate away from $\Sigma_0$, Lemma~\ref{phase-grad2}, asserts that for any $\zz\in G^k_l$ and $\theta\in(0,1)$ satisfying $G_{2\theta}(\zz)\subset G^k_l$, one has 
\begin{align}\label{Schauder-k}
\begin{aligned}
\theta^2\|D_v^2f_k\|_{L^\infty(G_\theta(\zz))} + \theta\|D_vf_k\|_{L^\infty(G_\theta(\zz))} 
&\lesssim \|f_k-f_k(\zz)\|_{L^\infty(G_{2\theta}(\zz))}\\
&\lesssim \theta^{1/2}[f_k]_{\Ck^{1/2}(G^k)}. 
\end{aligned}
\end{align} 
Besides, the function $f-f_k$ satisfies
\begin{align*}
\LL(f-f_k)=S + (\LL_0-\LL)f_k {\quad\rm in\ } G^k.
\end{align*} 
Applying the maximum principle, Lemma~\ref{max-k}, to $f-f_k$ in $G^k_l$ yields that  
\begin{align*}
\|f-f_k\|_{L^\infty(G^k_l)} \lesssim \|f-f_k\|_{L^\infty(\partial_{\rm kin}G^k_l)}  + 2^{-2k}\big(\|(\LL_0-\LL)f_k\|_{L^\infty(G^k_l)} + \|S\|_{L^\infty(G^k_l)}\big). 
\end{align*} 
By the definition of $G^k_l$, together with the condition that $f-f_k=0$ on $\partial_{\rm kin}G^k$, we obtain 
\begin{align*}
\|f-f_k\|_{L^\infty(\partial_{\rm kin}G^k_l)} \lesssim 2^{-l/2}[f-f_k]_{\Ck^{1/2}(G^k)}. 
\end{align*} 
In view of the H\"older continuity assumption on $A$ and \eqref{Schauder-k} with $\theta\approx2^{-l}$, we have 
\begin{align*}
\|(\LL_0-\LL)f_k\|_{L^\infty(G^k_l)} 
\lesssim \big(2^{3l/2-\alpha k} +2^{l/2}\big) [f]_{\Ck^{1/2}(G^k)} 
\lesssim 2^{3l/2-\alpha k} [f]_{\Ck^{1/2}(G^k)} . 
\end{align*} 
Gathering the above three estimates, we obtain
\begin{align*}
\|f-f_k\|_{L^\infty(G^k_l)} \lesssim \big(2^{-l/2} + 2^{3l/2-2k-\alpha k}\big) \|f\|_{\Ck^{1/2}(G^k)} + 2^{-2k}\|S\|_{L^\infty(G^k)}. 
\end{align*}
By taking $l=(1+\alpha/2)k$, we arrive at 
\begin{align}\label{hk2}
\|f-f_k\|_{L^\infty(G^k_l)} \lesssim 2^{-k/2-\alpha k/4}\|f\|_{\Ck^{1/2}(G^k)}  + 2^{-2k}\|S\|_{L^\infty(G^k)}. 
\end{align}

Now that $\LL_0(f_k-f_{k-1})=0$ in $G^k$, applying Lemma~\ref{max-k} and \eqref{hk2} yields that 
\begin{align*}
\|f_k-f_{k-1}\|_{L^\infty(G^k)}&= \|f_k-f_{k-1}\|_{L^\infty(\partial_{\rm kin}G^k)} = \|f-f_{k-1}\|_{L^\infty(\partial_{\rm kin}G^k)}\\
&\lesssim 2^{-k/2-\alpha k/4}\|f\|_{\Ck^{1/2}(G^k)}  + 2^{-2k}\|S\|_{L^\infty(G^k)}. 
\end{align*}
By using Lemma~\ref{phi-holder}, we obtain 
\begin{align}\label{hk3}
\begin{aligned}
\|f_k-f_{k-1}\|_{\Ck^{1/2}(G^{k+1})} 
&\lesssim 2^{k/2}\|f_k-f_{k-1}\|_{L^\infty(G^k)} \\
&\lesssim 2^{-\alpha k/4}\|f\|_{\Ck^{1/2}(G_R)}  + 2^{-3k/2}\|S\|_{L^\infty(G_R)}. 
\end{aligned}
\end{align} 
By the triangle inequality, for any fixed $k_0\in\mathbb{N}_+$, $k>k_0+1$ and $z\in G_l^{k+1}$, we have 
\begin{align*}
|f(z)|&\le |f_{k_0}(z)| + |f(z)-f_k(z)| 
+\Big|\sum\nolimits_{j=k_0+1}^k  (f_j-f_{j-1})(z)\Big|\\
&\le  [f_{k_0}]_{\Ck^{1/2}(G^{k_0})}\|z\|^{1/2} + |f(z)-f_k(z)| + \sum\nolimits_{j=k_0+1}^k [f_j-f_{j-1}]_{\Ck^{1/2}(G^{j+1})}\|z\|^{1/2}. 
\end{align*} 
Combining this with \eqref{hk1}, \eqref{hk2} and \eqref{hk3}, we derive   
\begin{align*}
|f(z)|\lesssim&\ 2^{k_0/2}\|f\|_{L^\infty(G_R)}\|z\|^{1/2} + 2^{-k/2-\alpha k/4}\|f\|_{\Ck^{1/2}(G_R)}  + 2^{-2k}\|S\|_{L^\infty(G_R)} \\
& + 2^{-\alpha k_0/4} \big(\|f\|_{\Ck^{1/2}(G_R)} + \|S\|_{L^\infty(G_R)}\big)\;\!\|z\|^{1/2} .
\end{align*} 
Let $k\in\mathbb{N}_+$ such that $k>k_0+1$ and $2^{-k-1}\le \|z\|\le 2^{-k}$. We see that for any $\|z\|\le 2^{-k_0}$, 
\begin{align*}
\|z\|^{-1/2} |f(z)|\lesssim 2^{-\alpha k_0/4}\|f\|_{\Ck^{1/2}(G_R)}
+ 2^{k_0/2}\|f\|_{L^\infty(G_R)} + \|S\|_{L^\infty(G_R)}, 
\end{align*}
whose left-hand side quantifies the $\Ck^{1/2}$ regularity of $f$ at $\Sigma_0$. In the light of Lemma~\ref{patching}, this estimate can be patched together with the one away from $\Sigma_0$ given by Proposition~\ref{CN-bdy}. Taking $k_0$ sufficiently large, we deduce that for any $0<r<R<1$ with $R-r\lesssim 2^{-k_0}$, 
\begin{align*}
\|f\|_{\Ck^{1/2}(G_r)} &\le \frac{1}{2}\;\!\|f\|_{\Ck^{1/2}(G_R)}
+ C2^{k_0/2}\|f\|_{L^\infty(G_R)} + C\|S\|_{L^\infty(G_R)}\\
& \le \frac{1}{2}\;\!\|f\|_{\Ck^{1/2}(G_R)}
+ C(R-r)^{-1/2}\|f\|_{L^\infty(G_R)} + C\|S\|_{L^\infty(G_R)}. 
\end{align*}
The desired result then follows from Lemma~\ref{technicallemma}. 
\end{proof}

\appendix
\section{Maximum principle}\label{app-max}
We recall the following standard a priori estimates as an adaptation of the classical maximum principle. The proof is included for completeness. 

Consider $G:=\Omega_{t,x}\times\Omega_v$ with Lipschitz domains $\Omega_{t,x}\subset\R\times\R^d$ and $\Omega_v\subset\R^d$, for which the kinetic boundary $G$ is given by  
\begin{align*}
\partial_{\rm kin} G=\{(t,x,v)\in\partial\Omega_{t,x}\times\Omega_v:(1,v)\cdot n_{t,x}<0\}\cup(\Omega_{t,x}\times\partial\Omega_v), 
\end{align*}
where $n_{t,x}\in\R^{1+d}$ is the unit outward normal vector at $(t,x)\in\partial\Omega_{t,x}$. 

\begin{lemma}\label{max-su}
Let $G=\{t_0-\tau<t\le t_0\}\cap G$ for some $\tau>0$. Suppose $f\in C^0(\overline{G})$ satisfies that $(\partial_t+v\cdot\nabla_x)f$ and $D_v^2 f$ are continuous in $G$, and 
\begin{align*}
\left\{\begin{aligned}
\, &\LL f\le 0 {\quad\rm in\ }G,\\
\, &\; f\le0 {\quad\rm on\ }\partial_{\rm kin} G. 
\end{aligned}\right. 
\end{align*}
Then, we have
\begin{align*}
f\le 0 {\quad\rm in\ }G.  
\end{align*}
\end{lemma}

\begin{proof}
Let $\varepsilon>0$. Consider the function 
\begin{align*}
f_\varepsilon(t,x,v):=f(t,x,v)- \varepsilon\,(t-t_0+\tau)
\end{align*} 
Here we remark that $t-t_0+\tau\in(0,\tau]$ whenever $(t,x,v)\in G$. We observe
\begin{align}\label{fep1}
\left\{\begin{aligned}
\,&\LL f_\varepsilon =-\varepsilon  {\quad\rm in\ } G,\\
\,&\; f_\varepsilon \le 0 {\quad\rm on\ } \partial_{\rm kin} G.
\end{aligned}\right. 
\end{align}
One argues by contradiction to establish that
\begin{align}\label{fep2}
f_\varepsilon \le 0 {\quad\rm in\ } G. 
\end{align}
Suppose, to the contrary, that the function $f_\varepsilon$ achieves its maximum over $\overline{G}$ at some point $z_1\in\overline{G}\backslash\partial_{\rm kin} G$. Then, 
\begin{align*}
\nabla_v f_\varepsilon(z_1) = 0 {\quad\rm and\quad} D_v^2f_\varepsilon(z_1)\le 0.
\end{align*}
Since $z_1\not\in\partial_{\rm kin} G$, we have 
\begin{align*}
(1,v)\cdot\nabla_{(t,x)}\,f_\varepsilon(z_1)
=\left(\partial_t+v\cdot\nabla_x\right)f_\varepsilon(z_1)\ge 0.
\end{align*}
It follows that 
\begin{align*}
\LL f_\varepsilon(z_1)\ge0, 
\end{align*}
which contradicts the first inequality in \eqref{fep1}. Therefore, \eqref{fep2} holds. Finally, sending $\varepsilon\to0$ in \eqref{fep2}, we obtain the desired result. 
\end{proof}

\begin{lemma}\label{max-k}
Let $G=\{t_0-\tau<t\le t_0\}\cap G$ for some $\tau>0$. For any solution $f$ to \eqref{KFP} subject to \eqref{Elliptic} in $G$ with $f|_{\partial_{\rm kin} G}=0$, we have
\begin{align*}
\sup\nolimits_{G}|f|\le \sup\nolimits_{\partial_{\rm kin} G}|f|+ \tau\sup\nolimits_{G}|S| {\quad\rm in\ }G.  
\end{align*}
\end{lemma}

\begin{proof}
Consider the function 
\begin{align*}
g_\pm(t,x,v):=\pm f(t,x,v)- (t-t_0+\tau)\sup\nolimits_{G}|S|-\sup\nolimits_{\partial_{\rm kin} G}|f|, 
\end{align*}
where $t-t_0+\tau\in(0,\tau]$ in $G$. It satisfies
\begin{align*}
\left\{\begin{aligned}
\,&\LL g_\pm =\pm S -\sup\nolimits_{G}|S|\le0 {\quad\rm in\ } G,\\
\,&\; g_\pm \le 0 {\quad\rm on\ } \partial_{\rm kin} G.
\end{aligned}\right. 
\end{align*}
It then follows from the definition and Lemma~\ref{max-su} that
\begin{align*}
\pm f- \tau\sup\nolimits_{G}|S|-\sup\nolimits_{\partial_{\rm kin} G}|f| \le g_\pm\le 0 {\quad\rm in\ } G, 
\end{align*}
which implies the desired estimate. 
\end{proof}

\section{Patching lemma}\label{app-patch}
There is a well-established general standard machinery for deriving global estimates by patching together interior and boundary estimates, which we often invoke implicitly when needed. For a precise statement of this methodology, we refer to \cite[Theorem~3.1]{Wang2} and present the following reformulation for our purposes. 

\begin{lemma}\label{patching}
Let $\mathcal{S}$ be a subset of the space of continuous functions defined on the closure of a bounded domain $G\subset\R^{1+2d}$. Fix $k\in\mathbb{N}$ and $\alpha\in(0,1)$. Assume the following three properties holds. 
\begin{enumerate}[label=(\roman*), leftmargin=*]
\item\label{patch-b} There is some constant $A_b>0$ such that for any $f\in\mathcal{S}$ and $\zz\in\partial G$, there exists a polynomial $p_\zz(z)$ of kinetic degree $k$ satisfying, for any $z\in\overline{G}$, 
\begin{align*}
&\|p_\zz\|_{\Ck^k(G)}\le A_b,\\
&|f(z)-p_\zz(z)|\leq A_b\;\!\|\zz^{-1}\circ z\|^{k+\alpha}. 
\end{align*}
\item\label{patch-in} There is some $A_0>0$ and $A_1,A_2\ge1$ such that for any $f\in\mathcal{S}$ and $\zz\in G$, there exists a polynomial $p_\zz$ of kinetic degree $k$ satisfying, for any $r\in(0,\theta]$ and $Q_{2\theta}(\zz)\subset G$, 
\begin{align*}
&\sum\nolimits_{|l|\le k}\nolimits \theta^l\|D^lp_\zz\|_{L^\infty(Q_\theta(\zz))}\le A_1\|f\|_{L^\infty(Q_{2\theta}(\zz))},\\
&\,\|f-p_\zz\|_{L^\infty(Q_r(\zz)} \le \big(A_2\;\!\theta^{-k-\alpha}\|f\|_{L^\infty(Q_{2\theta}(\zz))}+A_0\big)\;\! r^{k+\alpha}.
\end{align*}
\item\label{patch-inva} For every $f\in\mathcal{S}$ and $\zz\in\partial G$, if $p_\zz$ is the polynomial given in \ref{patch-b}, then $f-p_\zz$ satisfies the estimates stated in \ref{patch-in}. 
\end{enumerate} 
Then, we have $\mathcal{S}\subset\Ck^{k+\alpha}(G)$, and there is some constant $C_k>0$ depending only $k$ such that
\begin{align*}
\|f\|_{\Ck^{k+\alpha}(G)} \leq C_k\left(A_1A_2A_b+A_0\right). 
\end{align*}
\end{lemma}

\begin{proof}
It suffices to establish the estimate at interior points of $G$; otherwise, if $\zz\in\partial G$, property~\ref{patch-b} yields the desired bound directly. We hence focus on $\zz\in G$ for which $Q_{2\theta}(\zz)\subset G$ for some $\theta\approx\|\zz^{-1}\circ z_0\|$ and some $z_0\in\partial G$. 

Applying property~\ref{patch-b} at $z_0$ produces a polynomial $p_{z_0}$ such that   
\begin{align}\label{fpz0}
\|f-p_{z_0}\|_{L^\infty(Q_{2\theta}(\zz))}\lesssim A_b\;\!\theta^{k+\alpha}. 
\end{align}
It then follows from properties~\ref{patch-inva} and \ref{patch-in} at $\zz$ that there exists a polynomial $p_\zz$ such that for any $r\in(0,\theta]$, 
\begin{align*}
\|f-p_{z_0}-p_\zz\|_{L^\infty(Q_r(\zz))} 
&\le \big(A_2\;\!\theta^{-k-\alpha}\|f-p_{z_0}\|_{L^\infty(Q_{2\theta}(\zz))}+A_0\big)\;\!r^{k+\alpha}\\
&\lesssim (A_2A_b +A_0)\;\! r^{k+\alpha}. 
\end{align*}
For the complementary case $r\ge\theta$, we know from property~\ref{patch-b} that 
\begin{align*}
&\|f-p_{z_0}\|_{L^\infty(Q_r(\zz))}\lesssim A_b\;\!r^{k+\alpha},\\
&\|f-p_{z_0}\|_{L^\infty(Q_{2\theta}(\zz))}\lesssim A_b\;\!\theta^{k+\alpha}.
\end{align*}
In view of the estimate for $p_\zz$ from property~\ref{patch-in}, we have
\begin{align*}
\|p_\zz\|_{L^\infty(Q_r(\zz))} 
&\lesssim \sum\nolimits_{|l|\le k}\nolimits r^l\;\! \|D^lp_\zz\|_{L^\infty(Q_\theta(\zz))}\\
&\lesssim A_1\sum\nolimits_{|l|\le k}\nolimits r^l\;\! \theta^{-l}\;\!\|f-p_{z_0}\|_{L^\infty(Q_{2\theta}(\zz))}. 
\end{align*}
The combination of the above three estimates then implies that for $r\ge\theta$,  
\begin{align*}
\|f-p_{z_0}-p_\zz\|_{L^\infty(Q_r(\zz))} 
&\le \|f-p_{z_0}\|_{L^\infty(Q_r(\zz))} + \|p_\zz\|_{L^\infty(Q_r(\zz))} \\
&\lesssim A_b\;\!r^{k+\alpha} + A_1A_b\sum\nolimits_{|l|\le k}\nolimits r^l\;\!\theta^{k+\alpha-l}
\lesssim A_1A_b\;\!r^{k+\alpha}. 
\end{align*}
We thus conclude the proof. 
\end{proof}

\end{document}